\documentclass[11pt, one side, article]{memoir}

\settrims{0pt}{0pt} 
\settypeblocksize{*}{36.1pc}{*} 
\setlrmargins{*}{*}{1} 
\setulmarginsandblock{.98in}{.98in}{*} 
\setheadfoot{\onelineskip}{2\onelineskip} 
\setheaderspaces{*}{1.5\onelineskip}{*} 
\checkandfixthelayout

\usepackage{amsthm}
\usepackage{mathtools}

\usepackage[inline]{enumitem}
\usepackage{ifthen}
\usepackage[utf8]{inputenc} 
\usepackage{xcolor}

\usepackage[backend=biber, backref=true, maxbibnames = 10, style = alphabetic]{biblatex}
\usepackage[bookmarks=true, colorlinks=true, linkcolor=blue!50!black,
citecolor=orange!50!black, urlcolor=orange!50!black, pdfencoding=unicode]{hyperref}
\usepackage[capitalize]{cleveref}

\usepackage{tikz}

\usepackage{amssymb}
\usepackage{newpxtext}
\usepackage[varg,bigdelims]{newpxmath}
\usepackage{mathrsfs}
\usepackage{dutchcal}
\usepackage{fontawesome}
\usepackage{ebproof}
\usepackage{stmaryrd}
\usepackage{float}

  \crefformat{enumi}{\card#2#1#3}
  \crefalias{chapter}{section}

  \hypersetup{final}

  \setlist{nosep}
  \setlistdepth{6}

  \usetikzlibrary{ 
  	cd,
  	math,
  	decorations.markings,
		decorations.pathreplacing,
  	positioning,
  	arrows.meta,
  	shapes,
  }
  
  \tikzset{
biml/.tip={Glyph[glyph math command=triangleleft, glyph length=.95ex]},
bimr/.tip={Glyph[glyph math command=triangleright, glyph length=.95ex]},
}

\tikzset{
	tick/.style={postaction={
  	decorate,
    decoration={markings, mark=at position 0.5 with
    	{\draw[-] (0,.4ex) -- (0,-.4ex);}}}
  }
} 
\tikzset{
	slash/.style={postaction={
  	decorate,
    decoration={markings, mark=at position 0.5 with
    	{\node[font=\footnotesize] {\rotatebox{90}{$\sim$}};}}}
  }
}

\newcommand{\bifrom}[1][]{
	\begin{tikzcd}[ampersand replacement=\&, cramped]\ar[r, bimr-biml, "{#1}"]\&~\end{tikzcd}  
}

\theoremstyle{definition}
\newtheorem{definitionx}{Definition}[chapter]
\newtheorem{examplex}[definitionx]{Example}
\newtheorem{remarkx}[definitionx]{Remark}

\theoremstyle{plain}

\newtheorem{theorem}[definitionx]{Theorem}

\newtheorem{corollary}[definitionx]{Corollary}
\newtheorem{lemma}[definitionx]{Lemma}

\newtheorem*{theorem*}{Theorem}
\newtheorem*{proposition*}{Proposition}
\newtheorem*{corollary*}{Corollary}
\newtheorem*{lemma*}{Lemma}
\newtheorem*{warning*}{Warning}

\newenvironment{example}
  {\pushQED{\qed}\examplex}
  {\popQED\endexamplex}
  
 \newenvironment{remark}
  {\pushQED{\qed}\remarkx}
  {\popQED\endremarkx}
  
  \newenvironment{definition}
  {\pushQED{\qed}\definitionx}
  {\popQED\enddefinitionx}

\DeclareSymbolFont{stmry}{U}{stmry}{m}{n}
\DeclareMathSymbol\fatsemi\mathop{stmry}{"23}

\DeclareFontFamily{U}{mathx}{\hyphenchar\font45}
\DeclareFontShape{U}{mathx}{m}{n}{
      <5> <6> <7> <8> <9> <10>
      <10.95> <12> <14.4> <17.28> <20.74> <24.88>
      mathx10
      }{}
\DeclareSymbolFont{mathx}{U}{mathx}{m}{n}
\DeclareFontSubstitution{U}{mathx}{m}{n}
\DeclareMathAccent{\widecheck}{0}{mathx}{"71}

\DeclareMathOperator{\Hom}{Hom}

\DeclareMathOperator*{\colim}{colim}

\DeclareMathOperator{\ob}{Ob}

\newcommand{\cat}[1]{\mathcal{#1}}
\newcommand{\Cat}[1]{\mathbf{#1}}

\newcommand{\id}{\mathrm{id}}

\newcommand{\To}[2][]{\xrightarrow[#1]{#2}}

\newcommand{\Tto}[3][13pt]{\begin{tikzcd}[sep=#1, cramped, ampersand replacement=\&, text height=1ex, text depth=.3ex]\ar[r, shift left=2pt, "#2"]\ar[r, shift right=2pt, "#3"']\&{}\end{tikzcd}}

\newcommand{\from}{\leftarrow}

\newcommand{\From}[1]{\xleftarrow{#1}}

\newcommand{\card}{\,^{\#}}

\newcommand{\op}{^\tn{op}}

\newcommand{\tn}[1]{\textnormal{#1}}

\newcommand{\ul}[1]{\underline{#1}}

\newcommand{\nn}{\mathbb{N}}

\newcommand{\rr}{\mathbb{R}}

\newcommand{\smset}{\Cat{Set}}

\newcommand{\smcat}{\Cat{Cat}}
\newcommand{\lgcat}{\Cat{CAT}}

\newcommand{\ccatsharp}{\mathbb{C}\Cat{at}^{\sharp}}

\newcommand{\set}{\tn{-}\Cat{Set}}

\newcommand{\yon}{{\mathcal{y}}}
\newcommand{\poly}{\Cat{Poly}}

\newcommand{\tri}{\mathbin{\triangleleft}}

\newcommand{\qand}{\quad\text{and}\quad}
\newcommand{\qqand}{\qquad\text{and}\qquad}

\newcommand{\alg}{\tn{-}\Cat{Alg}}
\renewcommand{\O}{{\mathcal{O}}}

\renewcommand{\path}{\textit{path}}
\newcommand{\List}{{\textit{list}}}
\newcommand{\lott}{{\textit{lott}}}
\newcommand{\smc}{\textit{smc}}
\newcommand{\sm}{\textit{sm}}
\newcommand{\nerve}{\textit{nerve}}
\newcommand{\edge}{\mathrm{E}}
\newcommand{\vertex}{\mathrm{V}}

\newcommand{\EM}{\Cat{EM}}
\newcommand{\K}{\mathbb{K}}

\newcommand{\biglens}[2]{
     \begin{bmatrix}{\vphantom{f_f^f}#2} \\ {\vphantom{f_f^f}#1} \end{bmatrix}
}
\newcommand{\littlelens}[2]{
     \begin{bsmallmatrix}{\vphantom{f}#2} \\ {\vphantom{f}#1} \end{bsmallmatrix}
}
\newcommand{\lens}[2]{
  \relax\if@display
     \biglens{#1}{#2}
  \else
     \littlelens{#1}{#2}
  \fi
}

\newcommand{\scdots}[2][]{\mathinner{#1\overset{#2}{\cdots}#1}}

\newcommand{\comod}{\mathbb{C}\Cat{omod}}

\newcommand{\minus}{\mathbin{\tn{-}}}

\allowdisplaybreaks
\setsecnumdepth{section}
\settocdepth{section}
\settocdepth{chapter}

\begin{document}

\title{A Polynomial Construction of Nerves for Higher Categories}

\author{Brandon T. Shapiro \and David I. Spivak}

\date{\vspace{-.2in}}

\maketitle

\begin{abstract}
We show that the construction due to Leinster and Weber of a generalized Lawvere theory for a familially representable monad on a (co)presheaf category, and the associated ``nerve'' functor from monad algebras to (co)presheaves, have an elegant categorical description in the double category $\ccatsharp$ of categories, cofunctors, familial functors, and transformations. In $\ccatsharp$, which also arises from comonoids in the category of polynomial functors, both a familial monad and a (co)presheaf it acts on can be modeled as horizontal morphisms; from this perspective, the theory category associated to the monad is built using left Kan extension in the category of endomorphisms, and the nerve functor is modeled by a single composition of horizontal morphisms in $\ccatsharp$. For the free category monad $\path$ on graphs, this provides a new construction of the simplex category as $\Delta \coloneqq \lens{\path}{\path \circ \path}$. We also explore the free Eilenberg-Moore completion of $\ccatsharp$, in which constructions such as the free symmetric monoidal category monad on $\smcat$ can modeled using the rich language of polynomial functors.
\end{abstract}

\renewcommand\cftbeforechapterskip{2pt plus 1pt}

\begin{KeepFromToc}
\tableofcontents
\end{KeepFromToc}

\chapter{Introduction}

One of the main features of categorical algebra is its ability to exhibit often quite complicated algebraic structures in a multitude of different forms, giving us the choice of which form is most elegant or convenient for any given purpose. So successful is this project that most algebraic structures can be modeled in forms that are barely recognizable as ``algebraic'' at all, instead taking on a more topological or combinatorial flavor which can be quite useful for intuition, proofs, and computation. 

Lawvere in his thesis \cite{Lawvere:2004} introduced the notion of representing an algebraic theory as a category generated by a single object under finite products (now called a \emph{Lawvere theory}). For a finitary monad on the category of sets there is a corresponding Lawvere theory given by the opposite of a skeleton of the category of finitely generated free algebras, and an equivalence between algebras of the monad and product preserving functors from the Lawvere theory to $\smset$. Such functors to $\smset$, which we henceforth call \emph{copresheaves} on the Lawvere theory, can be interpreted as a combinatorial expression of the data of an algebra for the monad.

Categories have also long been studied as combinatorial objects using the fully faithful \emph{nerve} functor from $\smcat$ to the category of simplicial sets, namely contravariant functors from the simplex category $\Delta$ to $\smset$ (\emph{presheaves}), sending a category to its sets of length-$n$ sequences of adjacent morphisms equipped with functions for restricting, composing, and inserting identities into these sequences to change their length. This construction also has a topological interpretation, where a sequence of $n$ adjacent morphisms in a category along with the composites of all of its sub-sequences can be regarded as a topological $n$-simplex. Nerve functors have also been defined for various forms of higher categories such as $n$-cellular nerves of $n$-categories \cite{berger2002cellular} and dendroidal nerves of multicategories \cite{moerdijk2007dendroidal}, and are widely used to study higher dimensional algebraic objects using homotopy theory.

Leinster \cite{Leinster:blog} and Weber \cite{weber2007familial} unified many of these constructions\footnote{including those for (higher) categories and polynomial (but not all finitary) monads on $\smset$. Further generalizations were later given in \cite{berger2012monads,bourke2019monads}.} by defining for any \emph{familial} monad $m$ on a (co)presheaf category a \emph{theory category} $\Theta_m$ and a fully faithful nerve functor from $m$-algebras to presheaves on $\Theta_m$. Familial monads are cartesian monads whose algebras are (co)presheaves equipped with operations for composing any diagram of a fixed shape (\emph{arity}) into a single cell (see for instance \cite[Appendix C]{Leinster:2004a}), and include all polynomial monads on $\smset$, the free category monad on graphs, and monads whose algebras correspond to several different types of higher categories. In \cref{smc} we describe a familial monad on graphs whose algebras are symmetric monoidal categories.

The theory category $\Theta_m$ is defined as the full subcategory of $m$-algebras spanned by the free algebras on the arity diagrams, and the nerve functor is given by the restricted Yoneda embedding which identifies for each arity the set of such diagrams in an $m$-algebra. For the free category monad, the arities are the graphs $\vec n$ consisting of $n+1$ vertices and $n$ adjacent edges forming a directed path. The theory category is then the full subcategory of $\smcat$ spanned by the free categories on these graphs, which are the nonempty finite ordinal categories $[n]$ that form the objects of the simplex category.

This construction of theories and nerves, producing for a broad class of widely-studied algebraic structures a single choice of combinatorial representation, endows its outputs with an air of \emph{canonicity}. Embeddings of the same structures into categories of combinatorial objects with desireable properties can be alternatively achieved, such as by cubical nerves of categories, simplicial or $n$-simplicial nerves of $n$-categories, and truncated versions thereof. The choices of Lawvere theories of finitary polynomial monads, simplicial nerves of categories, cellular nerves of $n$-categories, and dendroidal nerves of multicategories, among others, are elevated in the categorical landscape by their association to the corresponding monads via this general construction. In other words (of Leinster in \cite{Leinster:blog}), it answers the question of ``Why is the definition [of the nerve] exactly that, not something slightly different?''

The purpose of this paper is to address the same question for Leinster and Weber's construction itself: in what sense is it canonical, and not easily replaced by something else? Is there a way of associating to a familial monad its theory category and nerve functor in a manner that is significant in the language of category theory? 

The answer we give lies in the formal category theory inside the double category $\ccatsharp$, whose objects are small categories, vertical morphisms are \emph{cofunctors}, and horizontal morphisms $c \bifrom d$ are given by the data underlying familial functors from the category of copresheaves on $d$ to that of copresheaves on $c$. This double category has rich categorical structure such as left Kan extensions, arising from its equivalent construction (see \cref{chap.catsharp}) as the double category of comonoids in the monoidal category $\poly$ of polynomial endofunctors on $\smset$ under composition, and familial monads arise as formal monads in $\ccatsharp$ (in the sense of \cite{Street:1972a} and \cite[Definition 5.4.1]{benabou1967introduction}). We can also model an algebra of a familial monad $m$ on $c$-copresheaves as a horizontal morphism $c \bifrom 0$ from the empty category, where $0$-copresheaves form the terminal category, equipped with a left $m$-action in $\ccatsharp$.

\begin{theorem*}[{\cref{coclosuretheory,nervecoalgebra,isnerve}}]
For $c$ a small category and $m$ a familial monad on $c$-copresheaves, the theory category $\Theta_m$ is dual to the category $\lens{m}{m \circ m}$ built in a canonical way (\cref{comonadcofunctor}) from the left Kan extension of the underlying familial functor $m$ along the composite $m \circ m$.

For $m(X) \to X$ an $m$-algebra, its nerve has the same underlying set of elements as $m(X)$ and the $\Theta_m$-presheaf structure on that set is determined in a canonical way (\cref{lenscoalgebras}) from a coalgebra structure on $m(X)$ built out of the algebra structure map of $X$ and the universal transformation associated to the left Kan extension $\lens{m}{m \circ m}$ as in \eqref{eqn.intro}.
\end{theorem*}

\begin{equation}\label{eqn.intro}
\begin{tikzcd}
c \ar[bimr-biml,""{name=S, below}]{rrr}{m} \ar[bimr-biml]{dr}[swap,outer sep=-1.5]{\lens{m}{m \circ m}} & & & c \rar[bimr-biml]{X} \dar[Rightarrow,shorten >=9,shorten <=3] & 0 \\
& c \rar[bimr-biml,""{name=T, above},swap]{m} & c \ar[bimr-biml]{ur}[swap,outer sep=-1.5]{m} \ar[bimr-biml,bend right=20]{urr}[swap]{X} & {}
\arrow[Rightarrow,shorten=4,from=S,to=T]
\end{tikzcd}
\end{equation}

These results demonstrate that the construction of theory categories is endowed with a universal property (\cref{coclosuretheory}) and provides yet another piece of evidence for Mac Lane's slogan that ``all concepts are Kan extensions.'' Moreover, using this description of $\Theta_m$ the nerve of an $m$-algebra $X$ admits (via the formal categorical result in \cref{lenscoalgebras}) the simple formal construction of \eqref{eqn.intro} that could apply in any double category of comonoids with left Kan extensions.

This formulation makes extensive use of the monoidal category $\poly$, particularly in the constructions of left Kan extensions (which arise from the right monoidal coclosure in $\poly$) and the category associated to a familial comonad such as $\lens{m}{m \circ m}$ in \cref{comonadcofunctor}. It also adds nerves of higher categories to the fast-growing repertoire of pure and applied mathematics which can be elegantly modeled using the category of polynomial functors, such as database theory \cite{spivak2021functorial}, open dynamical systems \cite{shapiro2022dynamic}, and double categorical copresheaves \cite[Example 5.20]{shapiro2023structures}.

We conclude with a discussion of the free completion under Eilenberg-Moore objects \cite{lack2002formal} of the underlying bicategory of $\ccatsharp$. The resulting bicategory has familial monads as objects, morphisms which correspond to a generalization of familial functors to categories of algebras, and 2-cells corresponding to all natural transformations between such functors. Algebras for formal monads in this bicategory, such as the free symmetric monoidal category monad on $\smcat$ in \cref{sm}, are famlial monad algebras with additional operations, and by formal results of Lack and Street in \cite{lack2002formal} these enhanced structures are always algebras for a ``composite'' ordinary familial monad (see \cref{smwreath}).

\section*{Notation}

The symbol $\sum$ denotes an indexed coproduct, the symbol $+$ denotes binary coproduct, and $0$ denotes an initial object. 

For $c$ a category, we write $c\set$ for the category of copresheaves on $c$, meaning functors $c \to \smset$. For $X$ a $c$-copresheaf and $C \in c(1)$ an object, we write $X_C$ for $X(C)\in\smset$.

\section*{Acknowledgments}

This material is based upon work supported by the Air Force Office of Scientific Research under award number FA9550-20-1-0348. We also appreciate the comments of our ACT2023 conference reviewers, in particular Reviewer uG3q.

\chapter{The Double Category $\ccatsharp$}\label{chap.catsharp}

We begin by recalling the construction of the double category $\ccatsharp$ using polynomial functors, then prove some structural results that will facilitate the description of higher categorical nerves in $\ccatsharp$.

\section{The category of polynomials}

In the theory of polynomial functors, endofunctors $\smset \to \smset$ which arise as coproducts of representables are often considered purely in terms of the combinatorial data of their representing sets, which formally resemble the exponents of a classical polynomial.

\begin{definition}
A \emph{polynomial} $p$ consists of a set $p(1)$ along with, for each element $I \in p(1)$, a set $p[I]$. We write
\[
p = \sum_{I \in p(1)} \yon^{p[I]}
\]
for such a polynomial. A morphism $\phi$ of polynomials $p \to q$ is a natural transformation. It can be cast set-theoretically as consisting of a function $\phi_1 \colon p(1) \to q(1)$ along with, for each $I \in p(1)$, a function $\phi^\#_I \colon p[I] \from q[\phi_1I]$. We write $\poly$ for the category of polynomials.
\end{definition}

Elements of the set $p(1)$ are called \emph{positions} of $p$, and for each $I\in p(1)$, elements of $p[I]$ are called \emph{directions}. A morphism $\phi$ is called \emph{cartesian} if each $\phi^\#_I$ is a bijection, and \emph{vertical} if $\phi_1$ is a bijection.

\begin{definition}[{\cite[Proposition 2.1.7]{spivak2021functorial}}]
We denote by $\yon$ the polynomial with a single position and a single direction. For polynomials $p,q$, their \emph{composition} is the polynomial
\[
p \tri q \coloneq \sum_{\substack{I \in p(1) \\ J \colon p[I] \to q(1)}} \yon^{\sum\limits_{i \in p[I]} q[Ji]}.\qedhere 
\]
\end{definition}

There is a monoidal structure on the category $\poly$ given by $(\yon,\tri)$.

\section{Comonoids and bicomodules}

\begin{definition}
A \emph{comonoid} in $\poly$ is a polynomial $c$ equipped with morphisms $\epsilon \colon c \to \yon$ (the counit) and $\delta \colon c \to c \tri c$ (the comultiplication) satisfying unit and associativity equations. A comonoid homomorphism is a morphism of polynomials $c \to c'$ which commutes with counits and comultiplications.
\end{definition}

\begin{definition}
For comonoids $c,d$ in $\poly$, a $(c,d)$-bicomodule is a polynomial $p$, called the \emph{carrier}, equipped with morphisms
\[
c \tri p \From{\lambda} p \To{\rho} p \tri d
\]
which commute with each other as well as the counits and comultiplications of $c$ and $d$, in the sense of \cite[Definition 2.2.11]{spivak2021functorial}. We will often denote a $(c,d)$-bicomodule $p$ as $c \bifrom[p] d$.
\end{definition}

In \cite[Corollary 2.2.10]{spivak2021functorial}, the author established using a theorem of Shulman \cite[Theorem 11.5]{shulman2008framed} that there is a double category $\comod(\poly)$ whose objects are comonoids and horizontal morphisms are bicomodules.

\begin{definition}\label{defcatsharp}
$\ccatsharp$ is the pseudo-double category $\comod(\poly)$ which has
\begin{itemize}
	\item as objects, the comonoids in $\poly$;
	\item as vertical morphisms, the comonoid homomorphisms;
	\item as horizontal morphisms from $c$ to $d$, the $(c,d)$-bicomodules;
	\item as squares between homomorphisms $\phi,\psi$ and bicomodules $p,p'$, the morphisms of polynomials $\gamma \colon p \to p'$ such that the diagram in \eqref{eqn.square} commutes;
	\begin{equation}\label{eqn.square}
	\begin{tikzcd}
	c \tri p \dar[swap]{\phi \tri \gamma} & p \lar \rar \dar{\gamma} & p \tri d \dar{\gamma \tri \psi} \\
	c' \tri p' & p' \lar \rar & p' \tri d'
	\end{tikzcd}
	\end{equation}
	\item as horizontal identities, the comultiplication bicomodules $c \tri c \From{\delta} c \To{\delta} c \tri c$; and
	\item as composition of horizontal morphisms $c \bifrom[p] d \bifrom[q] e$, the bicomodule $p \tri_d q$ on the top row of \eqref{eqn.biccomp},
	\begin{equation}\label{eqn.biccomp}
	\begin{tikzcd}
	c \tri (p \tri_d q) \dar & p \tri_d q \dar \rar \lar & (p \tri_d q) \tri e \dar \\
	c \tri p \tri q \dar[shift right=2] \dar[shift left=2] & p \tri q \dar[shift right=2] \dar[shift left=2] \rar \lar & p \tri q \tri e \dar[shift right=2] \dar[shift left=2] \\
	c \tri p \tri d \tri q & p \tri d \tri q \rar \lar & p \tri d \tri q \tri e
	\end{tikzcd}
	\end{equation}
	where each object in the top row of \eqref{eqn.biccomp} is computed as the equalizer of the column below it, using the fact that the functors $c \tri -$ and $- \tri e$ preserve connected limits, and the maps between them are induced by the underlying transformations between equalizer diagrams. This also shows how to horizontally compose squares, as a pair of adjacent squares provides the data of a transformation of equalizer diagrams which induces a map between the composite bicomodules.\qedhere
\end{itemize}
\end{definition}

\section{Categories, cofunctors, and familial functors}\label{sec.prafunctors}

The motivation for usng $\ccatsharp$ in this setting comes from recent results of Ahman--Uustalu \cite{ahman2014when} and Garner%
\footnote{The proof was originally sketched in \href{https://www.youtube.com/watch?v=tW6HYnqn6eI}{Garner's HoTTEST video}; see also \cite[Section 2.3]{spivak2021functorial}.}
 that, respectively, comonoids in $\poly$ can be interpreted as categories and that bicomodules between them correspond to familial functors between their associated copresheaf categories.

\begin{definition}\label{comonoidcategory}
For a polynomial comonoid $(c,\epsilon,\delta)$, its corresponding (small) category has
\begin{itemize}
	\item as objects, elements of the set $c(1)$;
	\item as morphisms out of an object $C \in c(1)$, the set $c[C]$;
	\item as codomain assignment for morphisms out of $C$, the function $\delta_1(c) \colon c[C] \to c(1)$;
	\item as composition of morphisms out of $C$, the function $\delta^\sharp_C \colon c[C] \times_{c(1)} c_\ast(1) \to c[C]$; and
	\item as the identity morphism at $C$, the function $\epsilon^\#_C \colon 1 \to c[C]$.\qedhere
\end{itemize}
\end{definition}

Conversely, given a small category $\cat{C}$, for any object $C\in\ob(\cat{C})$, let $\cat{C}[C]\coloneqq\sum_{C'\in\ob(\cat{C})}\cat{C}(C,C')$ denote the set of all morphisms with source $C$. The polynomial comonad corresponding to $\cat{C}$ is carried by the polynomial $c\coloneqq\sum_{C\in\ob(\cat{C})}\yon^{\cat{C}[C]}$. The counit map $\epsilon\colon c\to\yon$ consists of a choice of morphism out of each object, which we take to be the identity. The comultiplication map $\delta\colon c\to c\tri c$ consists of the data of codomains and composition (see the proof of \cite[Theorem 5.6]{shapiro2023structures}).

Comonoids in $\poly$ therefore correspond to small categories, though for convenience we will often refer to them as simply categories. However, homomorphisms of comonoids do not correspond to functors between categories.

\begin{definition}
For categories $c$ and $d$, a cofunctor $c \to d$ consists of a function from the objects of $c$ to those of $d$, along with a lift of every morphism out of an object $D$ in $d$ to a morphism out of every object in $c$ sent to $D$. These liftings must respect codomains, identities, and composites in $c$ and $d$.
\end{definition}

When $c$ and $d$ are regarded as polynomial comonoids, a cofunctor consists of a function $\phi_1 \colon c(1) \to d(1)$ along with, for each $C \in c(1)$, a function $c[C] \from d[\phi_1C]$; precisely a morphism of polynomials $\phi \colon c \to d$. Preservation of identities, codomains, and composites from the definition of a cofunctor corresponds to the morphism of polynomials preserving counit and comultiplication (see \cite[Corollary 5.12]{shapiro2023structures}).

\begin{definition}\label{defprafunctor}
A familially representable functor $F \colon d\set \to c\set$ (henceforth called simply a familial functor) is a functor of the form, for any $d$-copresheaf $X$ and object $C \in c(1)$,
\[
F(X)_C \cong \sum_{I \in p_C(1)} \Hom_{d\set}(p[I],X)
\]
where $p_{(-)}(1)$ is a functor $c \to \smset$ (which we will denote by simply $p(1)$), $p_C(1)$ is its evaluation at $C$, and $p[-]$ is a functor $\left(\int p(1)\right)\op \to d\set$ from the dual of the category of elements of $p(1)$.

When $p$ is a $(c,d)$-bicomodule and $C \in c(1)$, we have $p(1)\cong\sum_{C\in c(1)}p_C(1)$, and we recover $p_C(1)$ as the preimage of $C$ under the function $p(1) \To{\lambda(1)} (c \tri p)(1) \To{c\tri\,!} c(1)$. Moreover, for $I \in p_C(1)$ and $D \in d(1)$, the set $p[I]_D$ is the preimage of $D$ under the function $p[I] \to d(1)$ given by the element $1\To{I}p(1) \To{\rho(1)} (p \tri d)(1)$.
\end{definition}

Based on this interpretation, we will denote a $(c,d)$-bicomodule $p$ as
\[
\sum_{C \in c(1)} \sum_{I \in p_C(1)} \yon^{p[I]}
\]
where $p[I]$ has the structure of a $d$-copresheaf. The elements $I \in p_C(1)$ will be called \emph{operations} and the copresheaves $p[I]$ will be called \emph{arities}.


\begin{example}[Copresheaves as bicomodules]
For $c$ any category, a $(c,0)$-bicomodule $p$ is a familial functor from $0\set$, the terminal category, to $c\set$. The particular copresheaf on $c$ this functor picks out is $p(1)$, whose elements are positions of $p$ and whose $c$-copresheaf structure is determined by the map $c \tri p \From{\lambda} p$. As there is also a map $p \to p \tri 0$ which preserves the positions of $p$, and forces the polynomial $p$ to have an empty set of directions. The category of $(c,0)$-bicomodules and maps between them as in \eqref{eqn.square} with $\phi,\psi$ identities is equivalent to the category $c\set$.

The composition of a $(c,d)$-bicomodule and a $(d,0)$-bicomodule is precisely the $c$-copresheaf given by applying the familial functor $d\set \to c\set$ to a $d$-copresheaf.
\end{example}


\begin{example}[Identity bicomodules]
The identity bicomodule $c \bifrom[c] c$ has the form $\sum\limits_{C \in c(1)} \yon^{c[C]}$, so it has a single operation for each object of $c$ with arity the corepresentable copresheaf $c[C]$.
\end{example}

\begin{example}[Composite bicomodules]
For a $(c,d)$-bicomodule $p$ and a $(d,e)$-bicomodule $q$ of the forms
\[
p = \sum_{C \in c(1)}\sum_{I \in p_C(1)} \yon^{p[I]} \qqand q = \sum_{D \in d(1)}\sum_{J \in q_D(1)} \yon^{q[J]},
\]
their composite $(c,e)$-bicomodule has the form
\[
p \tri_d q = \sum_{C \in c(1)}\sum_{\substack{I \in p_C(1) \\ J \colon p[I] \to q(1)}} \yon^{\colim_{i \in p[I]} q[Ji]},
\]
where the colimit in the directions is indexed by the category of elements of the $d$-copresheaf $p[I]$ (see \cite[Propositions 3.11,3.12]{garner2018shapely}, \cite[Proposition 1.8]{shapiro2022thesis}).
\end{example}

\section{Comonads and comodules}

Both the construction of the theory category of a monad and the corresponding nerve functor from algebras of the monad to models of the theory can be constructed using general categorical constructions in $\ccatsharp$, which we now describe.

Recall that $\ccatsharp$ is defined as $\comod(\poly)$, where each object is a comonoid in $\poly$ and each horizontal morphism is a bicomodule. A general fact about the $\comod$ construction shows that a horizontal comonad in $\comod(\poly)$ can in fact be equivalently treated as an ordinary comonoid in $\poly$ via its carrier polynomial.

\begin{lemma}[Dual to {\cite[Remark 5.15]{Cruttwell.Shulman:2010a}}]\label{comonadcofunctor}
For $c$ a polynomial comonoid, given a $(c,c)$-bicomodule $e$ with the structure of a $(c,\tri_c)$-comonoid, the carrier polynomial $e$ forms a polynomial comonoid equipped with a homomorphism to $c$ which acts on positions by sending the elements of $e_C(1)$ to $C \in c(1)$. 
\end{lemma}

In fact, this correspondence between comonads on $c$ in $\ccatsharp$ and vertical morphisms into $c$ extends to relate comodules of a comonad to bicomodules into the corresponding comonoid.

\begin{lemma}\label{lenscoalgebras}
For a comonad $e$ on an object $c$ of $\ccatsharp$ and $d$ another object of $\ccatsharp$, there is a carrier-preserving equivalence between $(e,d)$-bicomodules and $(c,d)$-bicomodules with the structure of a left $e$-comodule.
\end{lemma}

In particular, when $d=0$ this means that if $X$ is a copresheaf on $c$ equipped with a coalgebra structure $X \to e(X)$ then there is a corresponding copresheaf on $e$ with the same set of elements, with each element of $X_C$ assigned an object in $e(1)$ by the function $X \to e \tri_c X \to e \tri_c 1 = e(1)$ which is sent to $C$ by the function $e(1) \to c(1)$.

\begin{proof}
Given a $(c,d)$-bicomodule $p$, the structure of a left $e$-comodule amounts to a morphism of polynomials $p \to e \tri_c p$, which by the definition of bicomodule composition induces a morphism $p \to e \tri p$, where the left $e$-comodule equations in $\ccatsharp$ imply the left $e$-comodule equations in $\poly$. Similarly, as $p \to e \tri_c p$ is a morphism of $(c,d)$-bicomodules, this $e$-coaction on $p$ in $\poly$ is compatible with the right $d$-coaction on $p$, so that $p$ has the structure of a $(e,d)$-bicomodule.

By \cref{comonadcofunctor}, $e$ has a homomorphism to $c$, and as such there is a square in $\ccatsharp$ as on the left in \eqref{eqn.cocartesian} from $e$ as the identity $(e,e)$-bicomodule to $e$ as a $(c,c)$-bicomodule given by the identity morphism on $e$. The $(e,d)$-bicomodule structure on $p$ induces a $(c,d)$-bicomodule structure where the left coaction is given by the composite $c \tri p \from e \tri p \from p$ and the $(c,d)$-bicomodule equations are guaranteed from the $(e,d)$-bicomodule equations as the map $c \from e$ is a comonoid homomorphism. The identity morphism on $p$ therefore provides a morphism of bicomodules as on the left in \eqref{eqn.cocartesian}. 
\begin{equation}\label{eqn.cocartesian}
\begin{tikzcd}
e \rar[bimr-biml,""{name=S, below}]{e} \dar & e \rar[bimr-biml,""{name=U, below}]{p} \dar & d \dar[equals] \\
c \rar[bimr-biml,""{name=T, above},swap]{e} & c \rar[bimr-biml,""{name=V, above},swap]{p} & d
\arrow[Rightarrow,shorten=4,from=S,to=T]
\arrow[Rightarrow,shorten=4,from=U,to=V]
\end{tikzcd}
\end{equation}
Note that both of these squares are cocartesian fillers (in the sense of equipments, see \cite[Lemma 2.3.13]{spivak2021functorial} and \cite{shulman2008framed}) of their respective vertical arrows and source horizontal arrows. This means that the cocartesian filler of the outer frame of the composite agrees with the square on the right, so by the universal property of cocartesian fillers the composite square induces a morphism of $(c,d)$-bicomodules from $p$ to $e \tri_c p$, providing the coaction of a left $e$-comodule structure on $e$ in $\ccatsharp$. The comodule equations are then straightforward to deduce from the left $e$-comodule structure of $p$ in $\poly$ and cofunctoriality of $e$ over $c$.
\end{proof}

\section{Right coclosure (left Kan extension)}

We now recall the main structure in $\ccatsharp$ which will allow us to model theories and nerves, the \emph{right coclosure} (also known as \emph{left Kan extension}).

\begin{definition}[{\cite[Proposition 2.4.6]{spivak2021functorial}}]\label{coclosure}
For a $(d,e)$-bicomodule $q$, the functor $- \tri_d q$ from $(c,d)$-bicomodules to $(c,e)$-bicomodules has a left adjoint $\lens{-}{q}$. For a $(c,e)$-bicomodule $p$ its carrier is defined to be
\begin{equation}\label{eqn.coclosure}
\lens{p}{q} \coloneq \sum_{C \in c(1)}\sum_{I \in p_C(1)} \yon^{q \tri_e p[I]},
\end{equation}
where $p[I]$ is regarded as an $(e,0)$-bicomodule.
\end{definition}

We note the unit and counit of this adjunction for convenience:
\begin{equation}\label{eqn.coclosure_unit_counit}
  p\to\lens{p}{q}\tri q
  \qqand
	\lens{r \tri q}{q}\to r
\end{equation}
The former illustrates how the right coclosure from \eqref{eqn.coclosure} corresponds to the left Kan extension, equivalently in $\ccatsharp$ and the bicategory of copresheaf categories and familial functors.
\[
\begin{tikzcd}
	c\ar[r,bimr-biml, "p", ""' name=p]&
	e\ar[d,biml-bimr, "q"]\\&
	d\ar[ul, bend left, biml-bimr, pos=.8, "\lens{p}{q}", ""' name=cocl]
	\ar[from=p, to=p|-cocl.south, shorten >=-8pt, Rightarrow]
\end{tikzcd}
\hspace{1in}
\begin{tikzcd}
	e\set\ar[d, "q\tri_d-"']\ar[r, "p\tri_e-", ""' name=p]&
	c\set\\
	d\set\ar[ur, bend right, "\tn{Lan}_pq"', "" name=cocl]
	\ar[from=p, to=p|-cocl.south, shorten >=0pt, Rightarrow]
\end{tikzcd}
\]

We will be particularly interested in the case of right coclosures with the structure of a comonad.

\begin{lemma}\label{thetacomonoid}
For $m$ a monad on $c$ in $\ccatsharp$ and $p$ any $(d,c)$-bicomodule, the left Kan extension $\lens{p}{p \tri_c m}$ is a comonad on $d$.
\end{lemma}

\begin{proof}
In the interest of space we construct only the counit and comultiplication transformations, but as they are derived directly from the unit and multiplication transformations of the monad $m$ it is straightforward to check that their counitality and coassociativity follow directly from unitality and associativity for $m$, respectively.

The strategy here is to use the universal property of the coclosure that morphisms of $(d,d)$-bicomodules from $\lens{p}{p \tri_c m}$ to $q$ are in bijective correspondence with morphisms from $p$ to $q \tri_d p \tri_c m$. The counit transformation $\lens{p}{p \tri_c m} \to d$ then corresponds to the transformation $p \to p \tri_c m \cong d \tri_d p \tri_c m$ given by the unit of $m$.

Similarly, the comultiplication $\lens{p}{p \tri_c m} \to \lens{p}{p \tri_c m} \tri_d \lens{p}{p \tri_c m}$ corresponds to the composite 
\[
p \to \lens{p}{p \tri_c m} \tri_d p \tri_c m \to \lens{p}{p \tri_c m} \tri_d \left(\lens{p}{p \tri_c m} \tri_d p \tri_c m\right) \tri_c m \to \lens{p}{p \tri_c m} \tri_d \lens{p}{p \tri_c m} \tri_d p \tri_c m
\]
consisting of two applications of the unit of the adjunction $\lens{-}{p \tri_c m} \dashv - \tri_d p \tri_c m$ followed by the multiplication of $m$.
\end{proof}

As a special case, when $m = c$ this shows that $\lens{p}{p}$ is a comonoid for any bicomodule $p$.

\chapter{Familial monads}

The main objects we consider are \emph{familial monads}, which can be regarded as encoding the data, structure, and properties of a particular type of higher (or lower) dimensional category.

\begin{definition}
A familial monad is a cartesian monad on a presheaf category whose underlying endofunctor is familial as in \cref{defprafunctor} (and hence arises from a formal monad in $\ccatsharp$).
\end{definition}

\begin{example}[Free category monad]\label{defpath}
We denote by $g$ the category $\edge \Tto{s}{t} \vertex$ whose copresheaves are precisely graphs, and by $\vec n$ the graphs with vertices $0,...,n$ and edges $i \minus 1 \to i$ for all $1 \le i \le n$. The free category monad on graphs classically sends a graph to the graph with the same vertices and edges given by finite directed paths in the original graph, which can be partitioned according to their length; in this way, we can familially represent this endofunctor $\path$ as (for a $g$-copresheaf $X$)
\[
\path(X)_\vertex \cong \Hom(\vec 0, X) \qquad\qquad \path(X)_\edge \cong \sum_{n \in \nn} \Hom(\vec n,X)
\]
as $\vec n$ is the ``walking length-$n$ path graph.'' The familial functor $\path$ then corresponds to the $(g,g)$-bicomodule 
\[
\path \coloneqq \{\vertex\}\yon + \{\edge\}\sum\limits_{n \in \nn} \yon^{\vec n},
\]
where the labels $\edge,\vertex$ indicate how the left coaction is defined on positions.
\end{example}

Following \cite[Section 5.2]{shapiro2022thesis} or unraveling the definition of a cartesian formal monad in $\ccatsharp$ under the correspondence in \cref{sec.prafunctors}, we get that a familial monad on the category $c\set$ consists of a $(c,c)$ bicomodule
\[
m = \sum_{C \in c(1)}\sum_{M \in m_C(1)} \yon^{m[M]}
\] 
(where $M \in m_C(1)$ is called an \emph{operation outputting a $C$-cell} and $m[M]$ is called the \emph{arity} of $M$) along with the following additional data subject to unit and associativity equations:
\begin{itemize}
	\item for each $C \in c(1)$, a \emph{unit} operation $\eta(C) \in m_C(1)$ and an isomorphism $m[\eta(C)] \cong c[C]$ between its arity and the representable copresheaf on $C$; and
	\item for each $M \in m_C(1)$ and $N \colon m[M] \to m(1)$ a morphism in $c\set$, a \emph{composite} operation $\mu(M,N) \in m_C(1)$ and an isomorphism $m[\mu(M,N)] \cong \colim\limits_{z \in m[M]} m[Nz]$ between its arity and the colimit of arities indexed by the category of elements of $m[M]$.
\end{itemize}

\begin{example}[Free category monad]
For $\path$, whose algebras are categories, the operations outputting edges correspond to the $n$-ary composition of arrows for each $n \in \nn$ and their arities are the graphs $\vec n$ with $n$ adjacent edges. The unit operation outputting an edge is the trivial 1-ary composition of a single edge, while the only operation outputting a vertex is the unit. The composite operations arise from ``plugging in'' $n$ different paths of possibly different lengths $m_1,...,m_n$ into the length-$n$ path to get a path of length $m_1+\cdots+m_n$ and the unique operation with that graph as its arity. This uniqueness of operations with respect to their arities is what ensures that the unit and associativity equations for categories hold among $\path$-algebras, and a similar property applies to other familial monads whose algebras are strict higher category structures.
\end{example}

\begin{example}[Symmetric operads]\label{operads}
Any symmetric operad which is $\Sigma$-free, meaning that the symmetric group actions are free on the sets of $N$-ary operations, has a corresponding familial monad on sets with the same algebras (see for instance \cite[Example C.2.5]{Leinster:2004a}). In polynomial notation, noting that the $(\yon,\yon)$-bicomodules corresponding to familial endofunctors on $\smset$ are precisely ordinary polynomials, a symmetric operad $\O$ corresponds to the endofunctor
\[
o \coloneqq \sum_{N \in \nn} \O_N \times \yon^{\ul N} = \sum_{(N,[O]) \in \sum_N \O_N/\Sigma_N} \yon^{\ul N}.
\]
The operations are given by the equivalence classes $[O]$ in the quotient set $\O_N/\Sigma_N$ and have arity $o[[O]] = \ul N = \{1,...,N\}$, and for each equivalence class $[O]$ we pick a representative $\bar O \in [O]$. The unique unit operation is given by the unit of $\O$. As $\colim\limits_{v \in o[[O]]} o[[Pv]] = \sum_v \ul{M_v} \cong \ul{M_1+\cdots+M_N}$ for $[O] \in \O_N/\Sigma_N$ and $[Pv] \in \O_{M_v}/\Sigma_{M_v}$, we can let $\mu([O],[P]) = [\bar O \circ \bar P]$ with the isomorphism 
\begin{equation}\label{eqn.operad}
\colim\limits_{v \in o[[O]]} o[[Pv]] = o[[P_1]] \coprod \cdots \coprod o[[P_N]] \cong o[[\bar O \circ \bar P]] \cong o[[\overline{O \circ P}]],
\end{equation}
where $\bar O \circ \bar P$ and $\overline{O \circ P}$ are by the equivariance law for symmetric operads related by a permutation in $\Sigma_{M_1+\cdots+M_N}$, and by $\Sigma$-freedom of $\O$ this permutation is unique and forms the rightmost isomorphism in \eqref{eqn.operad}. 
\end{example}

\section{Higher categories}

Many examples of familial monads whose algebras are higher categories are discussed in \cite[Sections 1.3 and 7]{shapiro2022thesis}; here we summarize some of those examples and provide a new one whose algebras are symmetric monoidal categories.

\begin{example}[Higher categories]
Various notions of algebraic higher categories are given by algebras for particular familial monads, including the following:
\begin{itemize}
	\item the free $n$-category monad on $n$-graphs (also known as $n$-globular sets), where there is a unique operation outputting a $k$-cell ($k \le n$) for each $k$-dimensional globular pasting diagram, which forms the arity of that operation \cite[Proposition F.2.3]{Leinster:2004a}; 
	\item the free double category monad on double graphs (which have two types of edges along with squares between them), where the operations outputting edges of each type and their arities resemble those for categories and there is a unique operation outputting a square for each pair of natural numbers $(n,m)$, with arity the double graph given by an $n \times m$ grid of squares \cite[Example 1.14]{shapiro2022thesis}; and
	\item the free (symmetric or nonsymmetric) multicategory monad on multigraphs,\footnote{We use the category theorists' convention that graphs are by default directed with loops and multiple edges, so that the term ``multigraph'' may refer to the data underlying a multicategory.} with an operation outputting an $n$-to-1 edge for each finite tree with $n$ leaves (and in the symmetric case a separate operation for each ordering of those $n$ leaves) \cite[Example 2.14]{weber2007familial}. 
\end{itemize}
\end{example}

\begin{example}[Symmetric monoidal categories]\label{smc}
We now define a familial monad $\smc$ on the category of graphs, for which the algebras are symmetric strict monoidal categories, whose underlying monoidal categories are strict but only symmetric up to coherent isomorphism. While we could just as well discuss general symmetric monoidal categories, and will indicate the necessary modifications to do so, symmetric strict monoidal categories will be a convenient example for algebraic familial monads in \cref{sec.algebraic}.\footnote{This example was previously explained by the second named author at \href{https://youtu.be/u8XCiI-ZSHc?t=10}{https://youtu.be/u8XCiI-ZSHc?t=10}.}

Let $g$ be the polynomial comonad that indexes graphs, as in \cref{defpath}. As a $(g,g)$-bicomodule, we define $\smc$ as
\[
\smc \coloneqq \{\vertex\}\sum_{N \in \nn} \yon^{\ul N} + \{\edge\}\!\!\!\sum_{\substack{N \in \nn,\; \sigma \in \Sigma_N \\ M_1,...,M_N \in \nn}} \yon^{\vec M_1 + \cdots + \vec M_N}
\]
where $\ul N$ denotes the graph with vertices labeled $1,...,N$ and no edges, and $\vec M_1 + \cdots + \vec M_N$ is the graph consisting of disjoint paths of lengths $M_1,...,M_N$ respectively. 
The graph $\smc(1)$ has $\nn$ as its set of vertices and $\sum\limits_{N \in \nn} \nn^N \times \Sigma_N$ as its set of edges, with both the source and target of $(N,\sigma,M_1,...,M_N)$ being $N$. We then need to specify two graph homomorphisms $\ul N \to \vec M_1 + \cdots + \vec M_N$ (which we will call source and target); the source inclusion maps the $N$ vertices to the source vertex of each of the $N$ paths, in order, while the target inclusion maps them to the target vertex of each of the paths with their order permuted by $\sigma$. 

The unit operations are then $1$ and $(1,\id,1)$ for vertices and edges respectively. A graph homomorphism $I \colon \ul N \to \smc(1)$ consists of $N$ natural numbers $I(1),...,I(N)$, and the composite $\mu(N,I) = \sum_i I(i) \in \smc_v(1)$ (this is the same composition as in the familial monad whose algebras are monoids). For a graph homomorphism $I \colon \vec M_1 + \cdots + \vec M_N \to \smc(1)$, note that as $\vec M_i$ is connected and $\smc(1)$ is a disjoint union of graphs with a single vertex and many loops, $I$ sends all of the vertices in $\vec M_i$ to a single vertex in $\smc(1)$, which we denote $I(i) \in \nn$. We further denote by $I(i,j)$ the image in $\smc_e(1)$ of the $j$th edge of $\vec M_i$, which consists of the natural number $I(i)$, a permutation $I(i,j)_\Sigma$, and $I(i)$-many natural numbers $I(i,j)_1,...,I(i,j)_{I(i)}$. The composite operation $\mu((N,\sigma,M_1,...,M_N),I)$ is then given by 
\[
\left(\sum_{i=1}^N I(i), \sum_{i=1}^N I(\sigma(i),1)_\Sigma \circ \cdots \circ I(\sigma(i),M_{\sigma(i)})_\Sigma, \right.
\]
\[
\left.\sum_{j = 1}^{M_1} I(1,j)_1,...,\sum_{j = 1}^{M_1} I(1,j)_{I(1)},...,\sum_{j = 1}^{M_N} I(N,j)_1,...,\sum_{j = 1}^{M_N} I(N,j)_{I(N)}\right),
\]
where a sum of permutations denotes the appropriately partitioned permutation on the sum of the underlying sets. Informally, the arity of the composite is formed by substituting the arity $\smc[I(i,j)]$ in for the $j$th edge of $\vec M_i$, and these arities will be compatible along vertices as the discrete graph $\smc[I(i)]$ is substituted in for each vertex in $\vec M_i$. The unit and associativity equations are then tedious but straightforward to check.

To see that algebras for $\smc$ are symmetric strict monoidal categories, observe that the strictly associative $n$-ary tensor product is defined on objects by the operation $N \in \smc_v(1)$ (0 being the unit) and on morphisms by the operation $(N,\id,1,...,1) \in \smc_e(1)$; in an algebra for $\smc$, these operations ensure that to any $N$ vertices there is an associated vertex corresponding to their tensor product, and likewise for $N$ morphisms. Composition and identities are defined by the operations $(1,\id,N) \in \smc_e(1)$ with the same arities as those for the free category monad $\path$, and the uniqueness of operations with the arities of $N$ morphisms in either series or parallel ensures the unitality and associativity of both composition and tensor product. Functoriality of the tensor product is witnessed by the equations
$$\mu\bigg((2,\id,1,1),I(1,1) = I(2,1) = (1,\id,2)\bigg) = (2,\id,2,2)$$
\vspace{-.5cm}$$ = \mu\bigg((1,\id,2),I(1,1) = I(1,2) = (2,\id,1,1)\bigg)\qand$$
\vspace{-.5cm}$$\mu\bigg((2,\id,1,1),I(1,1) = I(2,1) = (1,\id,0)\bigg) = (2,\id,0,0) = \mu\bigg((1,\id,0),I(1) = 2\bigg)$$
in $\smc_e(1)$ which encode that in an algebra the composition (resp. identity) of tensor products is the tensor product of compositions (resp. identities). Finally, the symmetry isomorphisms are of the form $(N,\sigma,0,...,0)$ and the source/target inclusions into their arities encode that they map from each $N$-ary tensor product to the tensor product of the same vertices in the appropriately permuted order. Naturality of the symmetry isomorphisms is encoded by the equation
$$\mu\bigg((1,\id,2),I(1,1) = (2,\sigma,0,0),I(2,1) = (2,\id,1,1)\bigg) = (n,\sigma,1,1)$$ 
\vspace{-.5cm}$$ = \mu\bigg((1,\id,2),I(1,1) = (2,\id,1,1),I(2,1) = (2,\sigma,0,0)\bigg)$$
in $\smc_e(1)$, while associativity and invertibility of symmetries are encoded by the fact that operations in $\smc_e(1)$ with arity $\ul N$ are uniquely determined by a single permutation in $\Sigma_N$.

This method of adding in higher dimensional invertible operations to enforce commutativity properties not up to equality but rather up to higher dimensional cells is a way to get around the $\Sigma$-free condition from \cref{operads} (and is explored in \cite{weber2015operads}). It can also be used in a similar manner to define a monad whose algebras are weakly unital and/or associative, such as general symmetric monoidal categories: to do so, one would modify the elements of $\smc_v(1)$ to consist not just of a natural number $N$ but also a parenthesization of $N$ elements (parenthesizations could be binary/nullary or arbitrary length, corresponding to biased or unbiased monoidal categories as in \cite[Sections 3.1,3.2]{Leinster:2004a}). Each operation in $\smc_e(1)$ as above would then also be equipped with two such parenthesizations of $N$ elements indicating its source and target operations, but these would not otherwise affect the data of $\smc$ which encodes the coherence property that any two parenthesizations of the same number of objects have a unique coherence isomorphism between their correpsonding tensor products. Further variations on the notion of monoidal category such as braided and/or lax monoidal categories can also be modeled in a fairly similar manner.
\end{example}

%
%

\section{Theory categories and Nerves}\label{sec.theories}

The fully faithful \emph{nerve} functor from categories to simplicial sets encodes the algebraic structure of a category in a combinatorial manner: whereas a graph is a rather simple copresheaf which obtains the additional structure of a category only as an algebra for the $\path$ monad, a simplicial set is a more complex copresheaf but can encode the structure of a category all on its own. In \cite[Section 4]{weber2007familial}, Weber showed that this type of construction can be generalized to algebras for any familial monad: for each such $m$, there is a category $\Theta_m$ and a fully faithful functor $N_m \colon m\alg \to \Theta_m\op\set$.

\begin{definition}[{\cite[Definition 4.4]{weber2007familial}}]\label{theorycat}
For a familial monad $m$ on the copresheaf category $c\set$, define its \emph{theory category} $\Theta_m$ to be the full subcategory of $m\alg$ spanned by the free algebras $m(m[M])$ for $M \in m_C(1)$ with $C \in c(1)$. Equivalently, $\Theta_m$ is the full subcategory of the Kleisli category of $m$ spanned by the objects $m[M]$ in $c\set$. We will denote its objects by simply $M$.

Given an $m$-algebra $A$, its nerve is the functor $N_m A \colon \Theta_m\op \to \smset$ sending $M$ to the set $\Hom_{c\set}(m[M],A)$ (or equivalently $\Hom_{m\alg}(m(m[M]),A)$).
\end{definition}

There is a functor $c\op \to \Theta_m$ sending $C \in c(1)$ to $\eta(C) \in m_C(1)$; we will sometimes treat this faithful functor as a subcategory inclusion and denote the object $\eta(C)$ of $\Theta_m$ as simply $C$.

\begin{theorem}[{\cite[Theorem 4.10]{weber2007familial}}]
The functor $N_m$ is fully faithful and its essential image is precisely those copresheaves $X$ satisfying the ``Segal'' condition 
\[
X_M \cong \lim_{z \in m[M]_C} X_C.
\]
\end{theorem}

We will sometimes call such copresheaves \emph{models} of the theory.

\begin{example}[Nerves of categories]
For the monad $\path$, the theory category $\Theta_\path$ has as objects the set $\{0\} + \nn$ of natural numbers with two distinct copies of 0 and as morphisms $k \to \ell$ Kleisli arrows $\vec k \to \path(\vec \ell)$, or equivalently functors $\path(\vec k) \to \path(\vec \ell)$, where the domain and codomain are the ordinal categories $[k]$ and $[\ell]$. Thus $\Theta_{\path}$ coincides with the simplex category $\Delta$, albeit with two isomorphic copies of the terminal category $[0]$,\footnote{This choice of category in the equivalence class of $\Delta$ admits a cofunctor from its opposite category to $g$, and is used in \cite[Section 6]{shapiro2022thesis} as a convenient convention for general theory categories.} 
corresponding to the nerve of a category being a simplicial set.
\end{example}

\begin{example}[Nerves of higher categories]\label{highernerves}
This construction also recovers the appropriate type of presheaves for the nerves of various types of higher categories:
\begin{itemize}
	\item the theory category for the free $n$-category monad is equivalent to Joyal's category $\Theta_n$, whose objects are free pasting diagrams of composable globular cells up to dimension $n$ \cite[Example 4.18]{weber2007familial}; 
	\item for the free double category monad it is $\Delta \times \Delta$ whose objects can be regarded as composable finite grids; and
	\item for the free (symmetric/nonsymmetric) multicategory monad it is the (ordinary/planar) dendroid category $\Omega$ whose objects are planar trees (in the symmetric case equipped with a permutation on the leaves) \cite[Example 2.14]{weber2007familial}. \qedhere
\end{itemize}
\end{example}

\begin{example}[Nerves of operad algebras]\label{operadtheory}
For a $\Sigma$-free symmetric operad $\O$, its theory category $\Theta_o$ is equivalent to the category of free $\O$-algebras with $N$ generators for each $N$ with $\O_N$ nonempty. When $\O$ has at least one operation in each arity, this is precisely the Lawvere theory associated to the finitary monad $o$ (or its opposite, depending on convention) \cite[Example 4.15]{weber2007familial}, and the nerve functor sends an $\O$-algebra to the corresponding model of the Lawvere theory in the classical sense. In this case $\Theta_o$ has finite coproducts and the Segal condition for copresheaves on $\Theta_o\op$ to be isomorphic to the nerve of an $\O$-algebra is precisely the product preservation condition for a model of the Lawvere theory as a functor $\Theta_o\op \to \smset$.
\end{example}


The category $\Theta_m$ has a wide subcategory $\Theta_m^0$ consisting of \emph{inert} morphisms $M \to M'$ of the form $m(f) \colon m(m[M]) \to m(m[M''])$ in the Eilenberg-Moore category or $\eta \circ f \colon m[M] \to m[M''] \to m(m[M''])$ in the Kleisli category. Equivalently, $\Theta_m^0$ is the full subcategory of $c\set$ spanned by the objects $m[M]$. There are functors
\[
c \To{i} \Theta_m^0 \To{j} \Theta_m,
\]
where $i$ is fully faithful and $j$ is identity on objects. The Segal condition for a $\Theta_m\op$-copresheaf $X$ can alternatively be expressed by the equation $j^\ast X \cong i_\ast(ji)^\ast X$, where $(-)^\ast$ denotes the restriction functor on copresheaves and $(-)_\ast$ denotes the functor on copresheaves given by right Kan extension (see \cite[Definition 4.4]{weber2007familial}). Indeed by definition, for an object $M$ of $\Theta_m$, $(j^\ast X)_M = X_M$ and $(i_\ast(ji)^\ast X)_M \cong \lim\limits_{z \in m[M]_C} X_C$.

\chapter{Theories and Nerves in $\ccatsharp$}\label{sec.nervesincatsharp}

The fact that theory categories are modeled in $\ccatsharp$ using the right coclosure follows from a more general relationship between the right coclosure and Kleisli categories.

\begin{theorem}\label{coclosurekleisli}
For $m$ a monad on $c$ in $\ccatsharp$ and $p$ any $(d,c)$-bicomodule, the carrier comonoid of the comonad $\lens{p}{p \tri_c m}$ on $d$ corresponds to the opposite of the full subcategory of the Kleisli category of $m$ spanned by the copresheaves on $c$ of the form $p[I]$, and the cofunctor from this category to $d$ sends the object $p[I]$ for $I \in p_D(1)$ to $D \in d(1)$.
\end{theorem}

\begin{proof}
Following Definitions \ref{comonoidcategory} and \ref{coclosure}, the objects of the category corresponding to $\lens{p}{p \tri_c m}$ are the elements $\sum_D p_D(1)$, denoted by $I \in p_D(1)$ for the appropriate choice of $D$, and the morphisms out of $I$ are given by the set 
\[
p \tri_c m \tri_c p[I] \cong \sum_{D' \in d(1)}\sum_{J \in p_{D'}(1)} \Hom_{c\set}\left(p[J],m \tri_c p[I]\right).
\]
It is straightforward to check from \label{thetacomonoid} that the codomain function sends an element of this set to $J \in p_{D'}(1)$; in other words, that the set of morphisms from $I$ to $J$ in this category consists of Kleisli morphisms for the monad $m$ from $p[J]$ to $p[I]$; and that composition and identities correspond to their analogues in the Kleisli category of $m$. The claim regarding the cofunctor to $D$ similarly follows directly from \cref{comonadcofunctor}.
\end{proof}

We can now conclude from \cref{coclosurekleisli} and \cref{theorycat} that the right coclosure in $\ccatsharp$ provides a construction of the theory category of a familial monad, and endows it with a universal property.

\begin{corollary}\label{coclosuretheory}
For $m$ a familial monad on $c\set$, the polynomial comonoid $\lens{m}{m \tri_c m}$ corresponds to the theory category $\Theta_m\op$ up to isomorphism. 

Therefore, $\Theta_m$ has the universal property that cofunctors $\Theta_m\op \to a$ commuting over $c$ for a cofunctor $a \to c$ correspond to transformations of familial endofunctors from $m$ to $a \tri_c m \tri_c m$ with respect to which the counit and comultiplication of $a$ regarded as a familial comonad commute with the unit and multiplication of the rightmost copy of $m$ in a suitable sense.\footnote{See the proof of \cref{nervecoalgebra} for a sense of what this looks like.}
\end{corollary}

As an example of this universal property, there is a cofunctor $\Theta_m\op \to c$ sending $M \in m_C(1)$ to $C \in c(1)$ which corresponds to the transformation $m \cong m \tri_c c \To{\id_m \tri \eta} m \tri_c m \cong c \tri_c m \tri_c m$. Another example is the vertical cofunctor $\Theta_m\op \to \left(\Theta_m^0\right)\op$ given by the bijective on objects functor $\Theta_m \from \Theta_m^0$, which can be modeled in $\ccatsharp$ using the following special case of \cref{coclosurekleisli}.

\begin{corollary}
For $p$ a $(d,c)$-bicomodule, category corresponding to the comonoid $\lens{p}{p}$ is the opposite of the full subcategory of $c\set$ spanned by the copresheaves of the form $p[I]$, over all $I\in p(1)$. 
\end{corollary}

Given this, under the universal property from \cref{coclosuretheory} this cofunctor $\lens{m}{m \tri_c m} \to \lens{m}{m}$ corresponds to the transformation $m \to \lens{m}{m} \tri_c m \To{\id \tri_c \eta} \lens{m}{m} \tri_c m \tri_c m$.


We are now ready to express the nerve functor in the language of $\ccatsharp$. The key idea behind this construction is that for an $m$-algebra $X$, the set of elements of its nerve $N_mX$ is the set $\sum\limits_{C \in c(1)}\sum\limits_{M \in m_C(1)} \Hom_{c\set}(m[M],X)$, which is precisely the set of elements of $m(X)$. The key feature of this construction in $\ccatsharp$ is therefore to exhibit this copresheaf on $c$ as a $\lens{m}{m \tri_c m}$-comodule, allowing it to lift to the structure of a model of the theory $\Theta_m$ using \cref{lenscoalgebras}.



\begin{theorem}\label{nervecoalgebra}
For $m$ a monad on $c$ in $\ccatsharp$, $p$ a $(d,c)$-bicomodule, and $x$ a $(c,b)$-bicomodule with the structure of a left $m$-module, the transformation of bicomodules in \eqref{eqn.nerve} endows $p \tri_c x$ with the structure of a left $\lens{p}{p \tri_c m}$-comodule.
\begin{equation}\label{eqn.nerve}
\begin{tikzcd}
d \ar[bimr-biml,""{name=S, below}]{rrr}{p} \ar[bimr-biml]{dr}[swap]{\lens{p}{p \tri_c m}} & & & c \rar[bimr-biml]{x} \dar[Rightarrow,shorten >=9,shorten <=3] & b \\
& d \rar[bimr-biml,""{name=T, above},swap]{p} & c \ar[bimr-biml]{ur}[description]{m} \ar[bimr-biml,bend right=20]{urr}[swap]{x} & {}
\arrow[Rightarrow,shorten=4,from=S,to=T]
\end{tikzcd}
\end{equation}
\end{theorem}

Note that when $b=0$, $x=X$ is an ordinary $m$-algebra $m(X) \to X$ and this construction generalizes the construction of the nerve on elements discussed above.

\begin{proof}
Let a $(c,b)$-bicomodule $x$ have an $m$-action morphism $\psi \colon m \tri_c x \to x$, and consider the coaction morphism $p \tri_c x \to \lens{p}{p \tri_c m} \tri_d p \tri_c x$ of $(d,b)$-bicomodules from \eqref{eqn.nerve}.

To show that $p \tri_c x$ is a comodule, we demonstrate that the coaction commutes with counit and comultiplication. To do so, we use the universal property of the coclosure, namely that a transformation $p \to q \tri_d p \tri_c m$ for $q$ a $(d,d)$-bicomodule, as on the left in \eqref{eqn.coclosureproperty}, factors uniquely through a transformation $\lens{p}{p \tri_c m} \to q$, as on the right in \eqref{eqn.coclosureproperty}.
\begin{equation}\label{eqn.coclosureproperty}
\begin{tikzcd}
d \ar[bimr-biml,""{name=S, below},shift left=1]{rrr}{p} \ar[bimr-biml, bend right=20]{dr}[swap]{q} & & & c \\
& d \rar[bimr-biml,""{name=T, above},swap]{p} & c \ar[bimr-biml]{ur}[description]{m} 
\arrow[Rightarrow,shorten=4,from=S,to=T]
\end{tikzcd}\qquad=\qquad\begin{tikzcd}
d \ar[bimr-biml,""{name=S, below},shift left=1]{rrr}{p} \ar[bimr-biml,""{name=U, below}]{dr}[inner sep=-1]{\lens{p}{p \tri_c m}} \ar[bimr-biml, bend right=40, shift right=2,""{name=V, above}]{dr}[swap]{q} & & & c \\
& d \rar[bimr-biml,""{name=T, above},swap]{p} & c \ar[bimr-biml]{ur}[description]{m} 
\arrow[Rightarrow,shorten=4,from=S,to=T]
\arrow[Rightarrow,shorten=-1,from=U,to=V]
\end{tikzcd}
\end{equation}


For the counit, the composite $p \tri_c x \to p \tri_c m \tri_c x \to p \tri_c x$ of the unit of $m$ with the $m$-action map of $x$, on the left in \eqref{eqn.counitcheck}, is the identity by the unit law for $m$-modules. Therefore the composite of the coaction with the counit of $\lens{p}{p \tri_c m}$ as defined in \cref{thetacomonoid}, on the right in \eqref{eqn.counitcheck}, is the identity as desired. 
\begin{equation}\label{eqn.counitcheck}
\begin{tikzcd}
d \ar[bimr-biml,""{name=S, below},shift left=1]{rrr}{p} \ar[bimr-biml, bend right=30]{dr}[swap]{d} & & & c \rar[bimr-biml]{x} \dar[Rightarrow,shorten >=9,shorten <=3] & b \\
& d \rar[bimr-biml,""{name=T, above},swap]{p} & c \ar[bimr-biml]{ur}[description]{m} \ar[bimr-biml,bend right=20]{urr}[swap]{x} & {}
\arrow[Rightarrow,shorten=4,from=S,to=T,"p \tri \eta"]
\end{tikzcd}\;\;=\quad\begin{tikzcd}
d \ar[bimr-biml,""{name=S, below},shift left=1]{rrr}{p} \ar[bimr-biml,""{name=U, below}]{dr}[inner sep=-1]{\lens{p}{p \tri_c m}} \ar[bimr-biml, bend right=40, shift right=2,""{name=V, above}]{dr}[swap]{d} & & & c \rar[bimr-biml]{x} \dar[Rightarrow,shorten >=9,shorten <=3] & b \\
& d \rar[bimr-biml,""{name=T, above},swap]{p} & c \ar[bimr-biml]{ur}[description]{m} \ar[bimr-biml,bend right=20]{urr}[swap]{x} & {}
\arrow[Rightarrow,shorten=4,from=S,to=T]
\arrow[Rightarrow,shorten=-1,from=U,to=V]
\end{tikzcd}
\end{equation}

For the comultiplication, the composite
\[
p \tri_c x \to \lens{p}{p \tri_c m} \tri_d \lens{p}{p \tri_c m} \tri_d p \tri_c \left(m \tri_c m\right) \tri_c x
\]
\[
\To{\cdots \tri_c \mu \tri_c x} \lens{p}{p \tri_c m} \tri_d \lens{p}{p \tri_c m} \tri_d p \tri_c m \tri_c x \To{\cdots \tri_c \psi} \lens{p}{p \tri_c m} \tri_d \lens{p}{p \tri_c m} \tri_d p \tri_c x
\]
on top in \eqref{eqn.comultcheck} agrees with the composite 
\[
p \tri_c x \to \lens{p}{p \tri_c m} \tri_d \lens{p}{p \tri_c m} \tri_d p \tri_c m \tri_c \left(m \tri_c x\right)
\]
\[
\To{\cdots \tri_c m \tri_c \psi} \lens{p}{p \tri_c m} \tri_d \lens{p}{p \tri_c m} \tri_d p \tri_c m \tri_c x \To{\cdots \tri_c \psi} \lens{p}{p \tri_c m} \tri_d \lens{p}{p \tri_c m} \tri_d p \tri_c x
\]
on the bottom in \eqref{eqn.comultcheck} by the multiplication law for $m$-modules, so the composite of the coaction with the comultiplication of $\lens{p}{p \tri_c m}$ agrees with the repeated application of the coaction as desired.
\begin{equation}\label{eqn.comultcheck}
\begin{tikzcd}
d \ar[bimr-biml,""{name=S, below},shift left=1]{rrrrr}{p} \ar[bimr-biml,""{name=U, below}]{dr}[swap]{\lens{p}{p \tri_c m}} & & & & & c \rar[bimr-biml]{x} \ar[Rightarrow,shorten >=-4,shorten <=3, shift left=3]{d} & b \\
& d \ar[bimr-biml,""{name=T, above},""{name=U, below},swap]{rrr}[description]{p} \ar[bimr-biml]{dr}[swap]{\lens{p}{p \tri_c m}} & & & c \ar[bimr-biml]{ur}[description]{m} \ar[Rightarrow,shorten >=36,shorten <=-5]{dr} & {} \\
& & d \rar[bimr-biml,""{name=V, above},swap]{p} & c \ar[bimr-biml]{ur}[description]{m} \ar[bimr-biml,bend right=30]{uurr}[swap]{m} \ar[bimr-biml,bend right=35]{uurrr}[swap]{x} & & {} 
\arrow[Rightarrow,shorten=4,from=S,to=T]
\arrow[Rightarrow,shorten=4,from=U,to=V]
\end{tikzcd}
\end{equation}
\[
=\begin{tikzcd}
d \ar[bimr-biml,""{name=S, below},shift left=1]{rrr}{p} \ar[bimr-biml,""{name=U, below}]{dr}[inner sep=-1]{\lens{p}{p \tri_c m}} \ar[bimr-biml, bend right=40, shift right=2,""{name=V, above}]{dr}[swap]{\lens{p}{p \tri_c m} \tri_c \lens{p}{p \tri_c m}} & & & c \rar[bimr-biml]{x} \dar[Rightarrow,shorten >=9,shorten <=3] & b \\
& d \rar[bimr-biml,""{name=T, above},swap]{p} & c \ar[bimr-biml]{ur}[description]{m} \ar[bimr-biml,bend right=20]{urr}[swap]{x} & {}
\arrow[Rightarrow,shorten=4,from=S,to=T]
\arrow[Rightarrow,shorten=-1,from=U,to=V]
\end{tikzcd}
\]
\end{proof}

Applying \cref{lenscoalgebras} when $b=0$, this makes $p \tri_c X$ a presheaf over the category $\Theta_m^p$ dual to $\lens{p}{p \tri_c m}$, which is the full subcategory of the Kleisli category of $m$ (or equivalently Eilenberg-Moore category) spanned by the copresheaves of the form $p[I]$ on $c$ (respectively, the free algebras $m(p[I])$). The nerve functor constructed at the level of generality in \cref{nervecoalgebra} is the right adjoint from $m$-algebras to presheaves over $\Theta_m^p$ induced by the inclusion of this subcategory into $m\alg$.

\begin{corollary}\label{isnerve}
In the case when $p=m$ and $b=0$, this assignment of a presheaf on $\Theta_m$ to an $m$-algebra agrees with the nerve given by Weber in \cite[Definition 4.4]{weber2007familial}. Furthermore, in the case of a general bicomodule $p$ this functor agrees with the evident generalization of Weber's nerve functor described above. 
\end{corollary}

\begin{proof}
We show that for general $p$ the copresheaf $p \tri_c X$ has $I$-cells given by maps $m \tri_c p[I] \to X$ of $m$-algebras and for a morphism $f \colon m \tri_c p[I'] \to m \tri_c p[I]$ of $m$-algebras, the corresponding function from $I$-cells to $I'$-cells is given by precomposition with $f$. 

Observe that on positions, the unit morphism $p \to \lens{p}{p \tri_c m} \tri_d p \tri_c m$ sends $I \in p(1)$ to the tuple
\[
\left(I \in p(1), (I' \in p(1))_{f \colon p[I'] \to m \tri_c p[I]}, \pi \circ f \colon p[I'] \to m \tri_c p[I] \to m(1) \right),
\]
where $f$ is a morphism of $c$-copresheaves, and on directions sends 
\[
(I',v \in p[I'], z \in m[\pi \circ f(v)])
\]
to $f(v,z)$. The coaction thus sends the pair $(I, p[I] \to X)$ to 
\[
(I, (I')_f, (p[I'] \to m \tri_c p[I] \to m \tri_c X \to X)_{f}).
\]
As a $\lens{p}{p \tri_c m}$-copresheaf then its $I$-cells are the maps $p[I] \to X$ of $c$-copresheaves, equivalently maps $m \tri_c p[I] \to X$ of $m$-algebras, and the application of a morphism $p[I'] \to m \tri_c p[I]$ to $p[I] \to X$ is the Kleisli composite, equivalently the composite of maps 
\[
m \tri_c p[I'] \to m \tri_c p[I] \to X
\]
of $m$-algebras as desired.
\end{proof}

While most common examples of nerves of higher categories are of the form $m \tri_c x$, generalizing to $p \tri_c x$ allows for more precise control over which patterns are probed for and how many times they are counted.

\begin{example}[Models of Lawvere theories for operad algebras]
For $\O$ a $\Sigma$-free symmetric operad and $\List$ the polynomial $\sum\limits_{N \in \nn} \yon^{\underline N}$ corresponding to the terminal \emph{non}symmetric operad, the category $\lens{\List}{\List \tri o}$ is the opposite of the full subcategory of $\O$-algebras spanned by the free $\O$-algebras on $N$ generators for all $N \in \nn$. This is precisely the Lawvere theory corresponding to the finitary monad associated to $\O$. As discussed in \cref{operadtheory}, the theory category $\Theta_o$, by comparison, may have many more objects (one for each operation in $O$ rather than one for each finite arity) and does not necessarily include all finite arities as $\O$ may have no $N$-ary operations for certain $N$. Hence the nerve functor associated to this category agrees precisely with the classical construction of a model for the Lawvere theory from an $o$-algebra. This construction also applies to any ordinary polynomial monad on $\smset$, where symmetric operads represent those which are cartesian.
\end{example}

\begin{remark}
Weber's Nerve Theorem \cite[Theorem 4.10]{weber2007familial} shows that the nerve functor is fully faithful, though we have not yet been able to express this proof in the language of $\ccatsharp$. It could also be interesting to consider which choices of $p$ admit an analogous result, or equivalently when $\Theta_m^p$ is dense in $m$-algebras.
\end{remark}

\begin{remark}
The Segal condition can be expressed in $\ccatsharp$, albeit in a slightly different form than stated at the end of \cref{sec.theories}. In \cite[Definition 4.4]{weber2007familial}, the Segal condition for $X$ is defined as the $(\Theta_m^0)\op$-copresheaf $j^\ast X$ being in the image of $i_\ast$, though in that case it will always be the image of specifically $(ji)^\ast X$. This version however is more convenient for working in $\ccatsharp$, where fully faithful functors and restrictions of copresheaves along them are more challenging to represent than functors such as $j^\ast$ and $i_\ast$. The Segal condition can then be expressed, at the level of generality of \cref{nervecoalgebra}, as the existence of a $c$-copresheaf $X'$ making the diagram below commute up to isomorphism in $\ccatsharp$. 
\[
\begin{tikzcd}[row sep=0, column sep=large]
& \ar[bimr-biml]{dr}{X} \lens{p}{p \tri_c m} \\
\lens{p}{p} \ar[bimr-biml]{ur}{\lens{p}{p \tri_c \eta}} \ar[bimr-biml]{dr}[swap]{p} & & 0 \\
& c \ar[bimr-biml]{ur}[swap]{X'}
\end{tikzcd}
\]
Composition with $\lens{p}{p \tri_c \eta}$ corresponds to restriction along the functor from the category of copresheaves of the form $p[I]$ to that of free $m$-algebras of the form $m(p[I])$, and thus agrees with $j^\ast$ when $p=m$. Similarly, when $p=m$ composition with the $\left(\lens{p}{p},c\right)$-bicomodule $p$ agrees with the right Kan extension $i_\ast$.
\end{remark}

\chapter{Algebraic familial monads}\label{sec.algebraic}

We now consider Lack and Street's construction of the free completion of $\ccatsharp$ (regarded as a bicategory) under Eilenberg-Moore objects, $\EM(\ccatsharp)$. This bicategory is very similar to that of formal monads in $\ccatsharp$ (as described in \cite{Street:1972a,lack2002formal}), but has a slightly adjusted definition of the 2-cells between monad morphisms. In $\ccatsharp$, this construction produces a bicategory which can be viewed (through the lens of the right $\tri$-coclosure) as having as objects familial monads, as morphisms functors between categories of algebras of familial monads which are themselves \emph{familial} in a suitable sense, and as 2-cells natural transformations between these functors. Monads in $\EM(\ccatsharp)$ include familiar constructions such as the free symmetric monoidal category monad on $\smcat$, and by a formal result of Lack and Street behave like an elegant generalization distributive laws in $\ccatsharp$.

\begin{definition}[{\cite[end of section 1]{lack2002formal}}]\label{EM}
The bicategory $\EM(\ccatsharp)$ has\footnote{While we wish to consider this construction in the setting of bicategories as discussed in \cite[Remark 1.1]{lack2002formal}, for convenience and readability we suppress the unitors and associators from our diagrams. Alternatively, as the underlying bicategory of $\ccatsharp$ is biequivalent to the sub-2-category of $\lgcat$ consisting of copresheaf categories, familial functors between them, and natural transformations, the construction for 2-categories can just as well be applied in that setting without modification.}
\begin{itemize}
	\item as objects, monads $(c,c \bifrom[m] c)$ in $\ccatsharp$;
	\item as morphisms, monad morphisms $(p,\alpha) \colon (c,m) \from (d,n)$ in $\ccatsharp$, where $p \colon c \bifrom d$ and $\alpha \colon m \tri_c p \Rightarrow p \tri_d n$ satisfy equations \eqref{eqn.monadmapunit} and \eqref{eqn.monadmapmult};
\begin{equation}\label{eqn.monadmapunit}
\begin{tikzcd}
c \dar[biml-bimr,swap]{c} & \lar[biml-bimr,""{name=S, below},swap]{c} c \dar{m} & d \lar[biml-bimr,""{name=U, below},swap]{p} \dar[biml-bimr]{n} \\
c & \lar[biml-bimr,""{name=T, above}]{c} c & d \lar[biml-bimr,""{name=V, above}]{p}
\arrow[Rightarrow,shorten=5,from=S,to=T,"\eta^m"]
\arrow[Rightarrow,shorten=5,from=U,to=V,"\alpha"]
\end{tikzcd} \qquad = \qquad \begin{tikzcd}
c \dar[biml-bimr,swap]{c} & \lar[biml-bimr,""{name=S, below},swap]{p} d \dar[biml-bimr]{d} & d \lar[biml-bimr,""{name=U, below},swap]{d} \dar[biml-bimr]{n} \\
c & \lar[biml-bimr,""{name=T, above}]{p} d & d \lar[biml-bimr,""{name=V, above}]{d}
\arrow[equals,shorten=5,from=S,to=T]
\arrow[Rightarrow,shorten=5,from=U,to=V,"\eta^n"]
\end{tikzcd}
\end{equation}
\begin{equation}\label{eqn.monadmapmult}
\begin{tikzcd}
c \dar[biml-bimr,swap]{m} & \lar[biml-bimr,""{name=S, below},swap]{c} c \ar[biml-bimr]{dd}{m} & d \lar[biml-bimr,""{name=W, below},swap]{p} \ar[biml-bimr]{dd}{n} \\
c \dar[biml-bimr,swap]{m} \\
c & \lar[biml-bimr,""{name=T, above}]{c} c & d \lar[biml-bimr,""{name=X, above}]{p}
\arrow[Rightarrow,shorten=10,from=S,to=T,"\mu^m"]
\arrow[Rightarrow,shorten=10,from=W,to=X,"\alpha"]
\end{tikzcd} \qquad = \qquad\begin{tikzcd}
c \dar[biml-bimr,swap]{m} & \lar[biml-bimr,""{name=S, below},swap]{p} d \dar[biml-bimr]{n} & d \lar[biml-bimr,""{name=W, below},swap]{d} \ar[biml-bimr]{dd}{n} \\
c \dar[biml-bimr,swap]{m} & \ar[biml-bimr,""{name=T, above},""{name=U, below}]{l}[description]{p} d \dar[biml-bimr]{n} \\
c & \lar[biml-bimr,""{name=V, above}]{p} d & d \lar[biml-bimr,""{name=X, above}]{d}
\arrow[Rightarrow,shorten=5,from=S,to=T,"\alpha"]
\arrow[Rightarrow,shorten=5,from=U,to=V,"\alpha"]
\arrow[Rightarrow,shorten=10,from=W,to=X,"\mu^n"]
\end{tikzcd}
\end{equation}
	\item as 2-cells, the variation on transformations of monad morphisms given by $\rho \colon (p,\alpha) \Rightarrow (q,\beta)$, where $\rho$ is a transformation $p \Rightarrow q \tri_d n$ satisfying equation \eqref{eqn.monadtransf};
\begin{equation}\label{eqn.monadtransf}
\begin{tikzcd}
c \dar[biml-bimr,swap]{m} & \lar[biml-bimr,""{name=S, below},swap]{p} d \dar[biml-bimr]{n} & d \lar[biml-bimr,""{name=W, below},swap]{d} \ar[biml-bimr]{dd}{n} \\
c \dar[biml-bimr]{c} & \ar[biml-bimr,""{name=T, above},""{name=U, below}]{l}[description]{p} d \dar[biml-bimr]{n} \\
c & \lar[biml-bimr,""{name=V, above}]{q} d & d \lar[biml-bimr,""{name=X, above}]{d}
\arrow[Rightarrow,shorten=5,from=S,to=T,"\alpha"]
\arrow[Rightarrow,shorten=5,from=U,to=V,"\rho"]
\arrow[Rightarrow,shorten=10,from=W,to=X,"\mu^n"]
\end{tikzcd} \qquad = \qquad \begin{tikzcd}
c \dar[biml-bimr,swap]{c} & \lar[biml-bimr,""{name=S, below},swap]{p} d \dar[biml-bimr]{n} & d \lar[biml-bimr,""{name=W, below},swap]{d} \ar[biml-bimr]{dd}{n} \\
c \dar[biml-bimr,swap]{m} & \ar[biml-bimr,""{name=T, above},""{name=U, below}]{l}[description]{q} d \dar[biml-bimr]{n} \\
c & \lar[biml-bimr,""{name=V, above}]{q} d & d \lar[biml-bimr,biml-bimr,""{name=X, above}]{d}
\arrow[Rightarrow,shorten=5,from=S,to=T,"\rho"]
\arrow[Rightarrow,shorten=5,from=U,to=V,"\beta"]
\arrow[Rightarrow,shorten=10,from=W,to=X,"\mu^n"]
\end{tikzcd}
\end{equation}
	\item identity morphisms are given by identities in $\ccatsharp$;
	\item composition of morphisms $(c,m) \From{(p,\alpha)} (d,n) \From{(q,\beta)} (e,o)$ is given by the composite morphism $p \tri_d q$ and the composite 2-cell in \eqref{eqn.monadmapcomp};
\begin{equation}\label{eqn.monadmapcomp}
\begin{tikzcd}
c \dar[biml-bimr,swap]{m} & \lar[biml-bimr,""{name=S, below},swap]{p} d \dar[biml-bimr]{n} & \lar[biml-bimr,""{name=U, below},swap]{q} e \dar[biml-bimr]{o} \\
c & \lar[biml-bimr,""{name=T, above}]{p} d & \lar[biml-bimr,""{name=V, above}]{q} e
\arrow[Rightarrow,shorten=5,from=S,to=T,"\alpha"]
\arrow[Rightarrow,shorten=5,from=U,to=V,"\beta"]
\end{tikzcd}
\end{equation}
	\item identity 2-cells $(p,\alpha) \Rightarrow (p,\alpha)$ for $(p,\alpha) \colon (c,m) \from (d,n)$ are given by the unit of $n$, $p \tri_d \eta^n \colon p \Rightarrow p \circ n$; 
	\item vertical composition of 2-cells $(p,\alpha) \xRightarrow{\rho} (q,\beta) \xRightarrow{\sigma} (r,\gamma)$ is given by the composite 2-cell on the left in \eqref{eqn.monadtransfvert};
\begin{equation}\label{eqn.monadtransfvert}
\begin{tikzcd}
c \dar[biml-bimr,swap]{c} & \lar[biml-bimr,""{name=S, below},swap]{p} d \dar[biml-bimr]{n} & d \lar[biml-bimr,""{name=W, below},swap]{d}] \ar[biml-bimr]{dd}{n} \\
c \dar[biml-bimr,swap]{c} & \ar[biml-bimr,""{name=T, above},""{name=U, below}]{l}[description]{q} d \dar[biml-bimr]{n} \\
c & \lar[biml-bimr,""{name=V, above}]{r} d & d \lar[biml-bimr,""{name=X, above}]{d}
\arrow[Rightarrow,shorten=5,from=S,to=T,"\rho"]
\arrow[Rightarrow,shorten=5,from=U,to=V,"\sigma"]
\arrow[Rightarrow,shorten=10,from=W,to=X,"\mu^n"]
\end{tikzcd}
\end{equation}
	\item and horizontal composition of 2-cells $(p,\alpha) \xRightarrow{\rho} (q,\beta)$ and $(r,\gamma) \xRightarrow{\sigma} (s,\delta)$ is given by either of the composite 2-cells in the center or right in \eqref{eqn.monadtransfhor}, which agree by \eqref{eqn.monadtransf}.
\begin{equation}\label{eqn.monadtransfhor}
\begin{tikzcd}
c \dar[biml-bimr,swap]{c} & \lar[biml-bimr,""{name=S, below},swap]{p} d \dar[biml-bimr]{d} & \lar[biml-bimr,""{name=Y, below},swap]{r} e \dar[biml-bimr]{o} & e \lar[biml-bimr,""{name=W, below},swap]{e} \ar[biml-bimr]{dd}{o} \\
c \dar[biml-bimr,swap]{c} & \ar[biml-bimr,""{name=T, above},""{name=U, below}]{l}[description]{p} d \dar[biml-bimr]{n} & \ar[biml-bimr,""{name=Z, above},""{name=A, below}]{l}[description]{s} e \dar[biml-bimr]{o} \\
c & \lar[biml-bimr,""{name=V, above}]{q} d & \lar[biml-bimr,""{name=B, above}]{s} e & e \lar[biml-bimr,""{name=X, above}]{e}
\arrow[equals,shorten=5,from=S,to=T]
\arrow[Rightarrow,shorten=5,from=U,to=V,"\rho"]
\arrow[Rightarrow,shorten=5,from=Y,to=Z,"\sigma"]
\arrow[Rightarrow,shorten=5,from=A,to=B,"\delta"]
\arrow[Rightarrow,shorten=10,from=W,to=X,"\mu^o"]
\end{tikzcd}\qquad=\qquad\begin{tikzcd}
c \dar[biml-bimr,swap]{c} & \lar[biml-bimr,""{name=S, below},swap]{p} d \dar[biml-bimr]{n} & \lar[biml-bimr,""{name=Y, below},swap]{r} e \dar[biml-bimr]{o} & e \lar[biml-bimr,""{name=W, below}]{e} \ar[biml-bimr]{dd}{o} \\
c \dar[biml-bimr,swap]{c} & \ar[biml-bimr,""{name=T, above},""{name=U, below}]{l}[description]{q} d \dar[biml-bimr]{d} & \ar[biml-bimr,""{name=Z, above},""{name=A, below}]{l}[description]{r} e \dar[biml-bimr]{o} \\
c & \lar[biml-bimr,""{name=V, above}]{q} d & \lar[biml-bimr,""{name=B, above}]{s} e & e \lar[biml-bimr,""{name=X, above}]{e}
\arrow[Rightarrow,shorten=5,from=S,to=T,"\rho"]
\arrow[equals,shorten=5,from=U,to=V]
\arrow[Rightarrow,shorten=5,from=Y,to=Z,"\gamma"]
\arrow[Rightarrow,shorten=5,from=A,to=B,"\sigma"]
\arrow[Rightarrow,shorten=10,from=W,to=X,"\mu^o"]
\end{tikzcd}
\end{equation}\qedhere
\end{itemize}
\end{definition}

\begin{definition}
We will write $m\alg_a$ for the category of bicomodules $c \bifrom[x] a$ equipped with the structure of a left $m$-module. 
The category of algebras for the corresponding familial monad on $c\set$ is given by $m\alg_0$. 

A monad morphism $(p,\alpha) \colon (c,m) \from (d,n)$ induces a functor $n\alg_a \to m\alg_a$ sending an $n$-module $d \bifrom[x] a$ to $p \circ x=p\tri_dx$, which has an $m$-module structure given by the composite in \eqref{eqn.inducedalg}.
\begin{equation}\label{eqn.inducedalg}
\begin{tikzcd}
c \dar[biml-bimr,swap]{m} & \lar[biml-bimr,""{name=S, below},swap]{p} d \dar[biml-bimr]{n} & \lar[biml-bimr,""{name=U, below},swap]{x} a \dar[biml-bimr]{a} \\
c & \lar[biml-bimr,""{name=T, above}]{p} d & \lar[biml-bimr,""{name=V, above}]{x} a
\arrow[Rightarrow,shorten=5,from=S,to=T,"\alpha"]
\arrow[Rightarrow,shorten=5,from=U,to=V]
\end{tikzcd}
\end{equation}
\end{definition}

While this diagram demonstrates formally that for $x$ an $n$-module $p \circ x$ is an $m$-module, we can use the right $\tri$-coclosure to unpack the precise form of this functor when $a=0$. By the universal property of the right coclosure, a 2-cell $\alpha \colon m \tri p \to p \tri n$ in $\ccatsharp$ corresponds to a 2-cell $\lens{m \tri p}{n} \to p$ which will illustrate how $p$ sends $m$-algebras to $n$-algebras in the style of familial functors. 

The data of a 2-cell $\lens{m \tri_c p}{n} \to p$ consists of a morphism $\alpha_1 \colon m(p(1)) \to p(1)$ of $c$-copresheaves along with natural maps in $d\set$, for each $M \in m(1)$ and $I \colon m[M] \to p(1)$ in $c\set$, of the form
\[
\alpha^\#_{M,I} \colon n\left(\colim_{z \in m[M]} p[I(z)]\right) \from p[\alpha_1(M,I \colon m[M] \to p(1))]
\]
which can be equivalently regarded as maps 
\[
n\left(\colim_{z \in m[M]} p[I(z)]\right) \from n\left(p[\alpha_1(M,I \colon m[M] \to p(1))]\right)
\]
of free $n$-algebras which commute with units and multiplications in a suitable sense. 

For an $n$-algebra $n(X) \to X$ in $d\set$, the $m$-algebra $p(X)=p\tri_dX$ is then given by
\[
p(X)_C \cong \sum_{I \in p_C(1)} \Hom_{d\set}\left(p[I],X\right) \cong \sum_{I \in p_C(1)} \Hom_{n\alg}(n(p[I]),X),
\]
for objects $C\in c(1)$. The definition of $m(p(X))$ unwinds to the copresheaf
\[\begin{aligned}
m(p(X))_C & \cong \sum_{M \in m_C(1)}\sum_{I \colon m[M] \to p(1)} \Hom_{d\set}\left(\colim_{z \in m[M]} p[I(z)],X\right) \\
& \cong \sum_{M \in m_C(1)}\sum_{I \colon m[M] \to p(1)} \Hom_{n\alg}\left(n\left(\colim_{z \in m[M]} p[I(z)]\right),X\right) \\
& \cong \sum_{M \in m_C(1)}\sum_{I \colon m[M] \to p(1)} \Hom_{n\alg}\left(\colim_{z \in m[M]} n\left(p[I(z)]\right),X\right),
\end{aligned}\]
and the $m$-algebra structure map $m(p(X)) \to p(X)$ acts by $\alpha_1$ on the indices of the sum and on each summand by precomposition with the algebra map 
\[
n\left(p[\alpha_1(M,I)]\right) \To{\alpha^\#_{M,I}} n\left(\colim_{z \in m[M]} p[I(z)]\right) \cong \colim_{z \in m[M]} n\left(p[I(z)]\right).
\]

These functors between categories of familial algebras therefore closely resemble familial functors between copresheaf categories, but with the additional data of an $m$-action which treats the arity $d$-copresheaves of $p$ as the corresponding free $n$-algebras, so as to lift the familial functor to map between categories of algebras. We therefore call them ``algebraic familial functors.'' 

\begin{definition}
For $m$ a familial monad on $c\set$ and $n$ a familial monad on $d\set$, an \emph{algebraic familial functor} is a functor $m\alg \from n\alg$ which commutes with an ordinary familial functor $c\set \From{p} d\set$ along the forgetful functors as in \eqref{eqn.algfam}.
\begin{equation}\label{eqn.algfam}
\begin{tikzcd}
m\alg \dar & \lar n\alg \dar \\
c\set \ar[phantom]{ur}{\circlearrowleft} & \lar{p} d\set
\end{tikzcd}
\end{equation}
In addition to the familial functor $c \bifrom[p] d$, the data of an algebraic familial functor consists of an $m$-algebra structure $m(p(X)) \to p(X)$ for all $n$-algebras $X$.
\end{definition}

As there must be a map $p(X) \to p(1)$ of $m$-algebras for all $n$-algebras $X$, we can deduce that the algebra structure on $p(X)$, whose component at the object $C$ of $c$ is a function
\[
\sum_{\substack{M \in m_C(1) \\ I \colon m[M] \to p(1)}} \Hom_{n\alg}\left(\colim_{z \in m[M]} n\left(p[I(z)]\right),X\right) \to \sum_{I \in p_C(1)} \Hom_{n\alg}(n(p[I]),X),
\]
is determined by a map $\alpha_1 \colon m(p(1)) \to p(1)$ in $c\set$ making $p(1)$ an $m$-algebra, along with for each $M \in m_C(1)$ a map 
\[
\alpha^\#_{M,I} \colon n\left(\colim_{z \in m[M]} p[I(z)]\right) \from n\left(p[\alpha_1(M,I)]\right)
\]
of free $n$-algebras such that $(p,\alpha_1,\alpha^\#)$ form a monad morphism from $n$ to $m$.

The arity $n$-algebras $n(p[I])$ and the maps between them imposed by $c$ have to be free so that the sets $\Hom_{n\alg}(n(p[I]),X)$ can be modeled by sets of the form $\Hom_{d\set}(p[I],X)$, which corresponds to the functor between categories of algebras descending to a well-defined functor between the underlying presheaf categories.

\begin{theorem}
Just as the bicategory $\ccatsharp$ is biequivalent to the 2-category of copresheaf categories, familial functors, and natural transformations, the $\EM(\ccatsharp)$ is biequivalent to the 2-category of copresheaf categories equipped with a familial monad, algebraic familial functors between categories of algebras, and natural transformations.
\end{theorem}

\begin{proof}
The 2-category of copresheaf categories, familial functors, and natural transformations is a sub-2-category of $\lgcat$, so applying the $\EM$ construction we get a sub-2-category of $\EM(\lgcat)$ as described in \cite[Section 2.1]{lack2002formal}. Therefore, following that description the objects are categories $c$ equipped with a familial monad $m$ on $c\set$, morphisms are functors between the corresponding categories of algebras which commute with a functor between the underlying copresheaf categories, and 2-cells are natural transformations of those functors. By \cite[Section 2.2]{lack2002formal} and the discussion above of the functors on categories of algebras induced by monad morphisms, the morphisms are precisely the algebraic familial functors.
\end{proof}

\begin{remark}
To see more clearly why the 2-cells of $\EM(\ccatsharp)$ are preferable to the simpler transformations of monad morphisms of \cite{Street:1972a}, consider the data of a natural transformation between algebraic familial functors $\rho \colon (p,\alpha),(q,\alpha) \colon n\alg \to m\alg$. Much like a transformation of ordinary familial functors, it consists of a morphism $\rho_1 \colon p(1) \to q(1)$ of $c$-copresheaves and natural maps from the arities of $q$ to those of $p$, but $\rho_1$ must commute with the $m$-actions on $p(1)$ and $q(1)$, the maps on arities can be any maps of $n$-algebras $\rho^\#_I \colon n(p[I]) \from n(q[\rho_1(I)])$, and these maps on arities commute with those of $\alpha,\beta$. In other words, from the definition of natural transformations of familial functors, replace the map of copresheaves on positions with a map of $m$-algebras, replace the maps of copresheaves on directions with maps of free $n$-algebras, and ensure that these maps on directions respect the algebra structure on the positions.

Without the algebra structure on $p(1),q(1)$ and the corresponding conditions on $\rho$, a morphism $\rho_1 \colon p(1) \to q(1)$ of $c$-copresheaves and natural morphisms $\rho\#_I \colon n(p[I]) \from n(q[\rho_1(I)])$ of $n$-algebras, or equivalently natural morphisms $n(p[I]) \from q[\rho_1(I)]$ of $d$-copresheaves, is precisely the data of a morphism $\lens{p}{n} \to q$ of $(c,d)$-bicomodules. Therefore, by the definition of the right-coclosure (\cref{coclosure}), it is also precisely the data of a morphism $p \to q \tri n$. 
\end{remark}

Before getting into nerves and higher categories, we begin with a more practical example.

\begin{example}[Normalizing a finite probability distribution]
Consider the monads $\List$ and $\lott$ corresponding to the terminal nonsymmetric operad and the nonsymmetric probability operad respectively; they are carried by
\[
\List\coloneqq\sum_{N:\nn}\yon^{\ul{N}}
\qqand
\lott\coloneqq\sum_{N:\nn, P:\Delta_N}\yon^{\ul{N}}
\]
where $\Delta_N$ is the set of probability distributions on $N$ elements. There is clearly a map $\alpha\colon\lott\to\List$, which forgets the probability distribution. Taking $c=\yon=d$, this gives a map of the form $(\yon,\alpha)\colon(\yon,\lott)\from(\yon,\List)$ in $\EM(\ccatsharp)$. For any polynomial monad $t$, let $\ul{t}$ denote the object $(\yon,t)$ in $EM(\ccatsharp)$.

The process of \emph{normalization} gives a map $(r,\nu)\colon\ul{\List}\from\ul{\lott}$ in $\EM(\ccatsharp)$, where $r\coloneqq\rr_{>0}\,\yon+1$. Indeed, it can be constructed as the composite of two maps
\[
\ul{\List}\From{(\yon+1,\rho)}\ul{\yon\List}\From{(r',\nu')}\ul{\lott}
\]
where $\yon\List$ is the monad of nonempty lists and $r'\coloneqq\rr_{>0}\,\yon$.

The map $\rho\colon\List\tri(\yon+1)\to(\yon+1)\tri\yon\List$ is given by removing all the exceptions from a list and returning the result if it is nonempty and an exception otherwise. This can be understood as the functor from semigroups to monoids given by adding a fresh identity element, say $e$. Given a semigroup $G$ and list, each of whose elements is either $e$ or an element of $G$, remove all the $e$'s and if the resulting list is empty, return $e$, otherwise return the product as defined by the semi-group structure.

The map $\nu'\colon\yon\List\tri r'\to r'\tri\lott$ is given on positions by
\[\nu(x_1,\ldots,x_n)\coloneqq \left(x,\left(\frac{x_1}{x},\ldots,\frac{x_n}{x}\right)\right),
\qquad\text{where}\quad
x\coloneqq x_1+\cdots+x_n
\]
and on directions by identity. It is easy to check that both satisfy \cref{eqn.monadmapunit,eqn.monadmapmult}.

The normalization map $(\rr_{>0}\,\yon+1,\nu)\colon\ul{\List}\from\ul{\lott}$ induces a map $\lott\alg_0\to\List\alg_0$. So for example, if $C$ is a convex set and hence has a $\lott$-algebra structure, there is an induced monoid with underlying set $C^\tn{cone}\coloneqq\rr_{>0}\times C+1$. Its unit is the distinguished (``cone'') point $1\in C^\tn{cone}$, and the product $(x,c)*(x',c')$ is $(x+x',\frac{x}{x+x'}c+\frac{x'}{x+x'}c')$. The notation $\frac{x}{x+x'}c+\frac{x'}{x+x'}c'$ is the use of the convex structure on $C$.
\end{example}

\begin{example}[The (inert) nerve functor]
The composite of the nerve functor from $m$-algebras to copresheaves on $\Theta_m\op$ with the forgetful functor to copresheaves on $(\Theta_m^0)\op$ is an algebraic familial functor given by the monad morphism $\left(\lens{m}{m},\id\right) \from (c,m)$
\[
\nerve^0 \coloneqq \sum_{M \in \lens{m}{m}(1)} \{M\}\yon^{m[M]}=\sum_{M \in m(1)} \{M\}\yon^{m[M]}
\]
As the target is an identity monad, there is no additional structure to specify. 
The need for this functor to land in copresheaves on $(\Theta_m^0)\op$ rather than $\Theta_m\op$ comes from the requirement that the maps between the arities $m[M]$ are maps of copresheaves on $c$ rather than the corresponding free $m$-algebras; $m[M]$ is functorial from $\Theta_m$ only in $m$-algebras, whereas when restricted to $\Theta_m^0$ it is functorial in $c\set$.
\end{example}

\begin{example}[Category of elements of a $c$-copresheaf]
For any category $c$ there is an algebraic familial functor $(g,\path) \from (c,\id)$ given by
\[
el \coloneqq c+c_\ast \cong \{\vertex\}\sum_{C \in c(1)} \yon^{c[C]} + \{\edge\}\!\!\!\!\!\sum_{C \in c(1), \;\; f \in c[C]} \yon^{c[C]},
\]
where the source and target of $(c,f)$ are respectively $c$ and the codomain of $f$ in $c$. Given a copresheaf $X$ on $c$, the familial functor to graphs associated to $el$ produces the graph whose vertices are all of the elements of $X$ and with an edge $v \to X_f(v)$ for each $v \in X_C$ and $f \in c[C]$.

$el(1)$ is the underlying graph of the category $c$, so the category structure on $c$ gives it a canonical choice of $\path$-algebra structure sending $(f_1,...,f_n)$ to $f_n \circ \cdots \circ f_1$. We then need to specify for any sequence of composable morphisms $C_0 \to C_1 \to \cdots \to C_n$ in $c$ a map
\[
\colim\left(c[C_0] = c[C_0] \from c[C_1] = c[C_1] \from \cdots = c[C_{n-1}] \from c[C_n]\right) \from c[C_0],
\]
but the colimit on the left agrees with $c[C_0]$ so this map can be taken to be an isomorphism. The resulting functor from $c\set$ to categories is precisely the category of elements functor.
\end{example}


We will now devote our attention to formal monads in $\EM(\ccatsharp)$, \emph{algebraic familial monads}, which act on the category of algebras for a familial monad by an algebraically familial endofunctor equipped with unit and multiplication transformations. This class of monads includes many monads on the category $\smcat$ which are common in the literature, though the requirement that the arity algebras are free excludes certain functors which would otherwise appear to be ``familial.''

\begin{example}[Free symmetric monoidal category on a category]\label{sm}
We now show how the free symmetric monoidal category monad on $\smcat$ is an algebraic familial monad on $\path$-algebras,  modeled as a monad on $(g,\path)$ in $\EM(\ccatsharp)$. The idea is that the underlying $(g,g)$-bicomodule $\sm$ will contain only the operations of a symmetric monoidal product and unit, as it is to be applied to a category which already has compositions and identities of morphisms (hence the notation $\sm$ which drops the ``c'' from $\smc$ which also encoded category structure). As in \cref{smc}, for simplicity we will consider symmetric strict monoidal categories, but the construction is easily modified to general symmetric monoidal categories in a similar manner, adding in operations for every binary association and coherences between them.

The $(g,g)$-bicomodule $\sm$ is defined as
\[
\sm \coloneqq \{\vertex\}\sum_{N \in \nn} \yon^{\ul N} + \{\edge\}\sum_{N \in \nn, \;\; \sigma \in \Sigma_N} \yon^{\vec 1 + \scdots{N} + \vec 1},
\]
where the requisite source and target inclusions $\ul N \to \vec 1 + \scdots{N} + \vec 1$ send a vertex $i \in \ul N$ to respectively the source vertex of the $i$th edge and the target vertex of the $\sigma(i)$th edge.

When regarded as a familial endofunctor on graphs, the operations outputting vertices produce a new vertex (the $N$-ary  tensor product) for each list of $N$ vertices, while the operations outputting edges extend the $N$-ary tensor product to edges while also adding symmetries with an edge to the tensor product of any permutation of the target vertices. When this endofunctor is applied to a category, applying the $(N,\sigma)$ operation to $N$ identity morphisms will produce the symmetry morphisms of a symmetric monoidal category, and applying the operation in general will produce tensor products of morphisms composed with symmetry morphisms.\footnote{The symmetry morphisms could be added separately as operations outputting edges with ariries $\ul N$, but then their composites with tensor products of morphisms would still need to be included to define the monad morphism structure.} The equations governing these operations when applied to a category will be imposed by the monad morphism structure on $\sm$.

Note that $\sm(1)$ is the graph with vertex set $\nn$ and, for each vertex $N$, a loop for each $\sigma \in \Sigma_N$. It has a $\path$-algebra structure $\alpha_1 \colon \path(\sm(1)) \to \sm(1)$ corresponding to the category of finite sets of the form $\ul N$ and permutations with the usual identities and compositions. In particular, for $I \colon \vec n \to \sm(1)$ where $I$ sends all vertices to $N \in \sm_v(1)$ and the $i$th edge to the permutation $\sigma_i \in \Sigma_N$, $\alpha_1(n,I) = (N,\sigma_n \circ \cdots \circ \sigma_1) \in \sm_e(1)$.

To exhibit $\sm$ as a monad morphism then requires morphisms of $\path$-algebras
\[
\alpha^\#_{n,I} \colon \path\left( \colim_{v \in \vec n} \sm[I(v)] \right) \from \path(p[N,\sigma_n \circ \cdots \circ \sigma_1]),
\]
for $I \colon \vec n \to \sm(1)$ as described above. Such a morphism amounts to a functor from $\path(\vec 1 + \scdots{N} + \vec 1)$, the disjoint union of $N$ walking arrow categories, to $\path(\vec n + \scdots{N} + \vec n)$, the disjoint union of $N$ copies of the ordinal category with $n$ adjacent generating morphisms (so $n+1$ objects). This functor the ``cocomposition'' functor given by sending each walking arrow to the maximal composite morphism of the corresponding ordinal. The monad morphism equations here are satisfied as $\alpha^\#_{1,I}$ is the identity functor on $\path(p[N,\sigma])$ and the composite $\path(\vec 1) \to \path(\vec n) \to \path(\vec m)$ is a cocomposition, where the first factor is a cocomposition functor and the second is a colimit of cocompositions.

The monad morphism structure $\alpha$ on $\sm$ encodes how $\sm$ sends a $\path$-algebra to a $\path$-algebra by indicating how to take $n$-ary composites in the free symmetric monoidal category on a category: when composing a sequence of $n$-many $N$-ary tensor products of morphisms each composed with a symmetry morphism, commute all of the symmetry morphisms to the end and compose them with the tensor products of the $N$ separate composite morphisms.

We now define a monad structure on $\sm$ in $\EM(\ccatsharp)$, where the unit and multiplication turn out to be ordinary transformations of monad morphisms rather than requiring the generalization thereof in \cref{EM}. This is not surprising though, as it will be the case for any cartesian monad whose unit and multiplication transformations are isomorphisms on arities and therefore do not need the flexibility of being Kleisli maps rather than maps on the underlying copresheaves. The unit $\eta^{\sm} \colon g \to \sm \To{\sm \tri \eta^{\path}} \sm \tri \path$ picks out the operations $1 \in \sm_v(1)$ and $(1,\id) \in \sm_e(1)$, and the multiplication $\mu^{\sm} \colon \sm \tri_g \sm \to \sm \To{\sm \tri \eta^{\path}} \sm \tri \path$ sends $(N;M_1,...,M_N)$ to $M_1 + \cdots + M_N$ and $((N,\sigma);(M_1,\sigma_1),...,(M_N,\sigma_N))$ to $(M_1 + \cdots + M_N,\sum_i \sigma_{\sigma(i)})$. It is straightforward to check that these assignments are unital and associative.

To see that this is indeed the free symmetric monoidal category monad requires the observation that to build a symmetric monoidal category from an ordinary category, it suffices to add in associative $N$-ary tensor products of all objects and morphisms and natural symmetry isomorphisms; this functor does so using the observation that any morphism is then uniquely a composite of a tensor product of morphisms with a symmetry isomorphism, as by naturality these can be commuted past each other allowing any composite of such morphisms to be factored into this form.
\end{example}

\begin{example}[Free categorical algebras of an operad]
The previous example can be generalized to produce a monad on $(g,\path)$ for any symmetric operad $\O$, inspired by the construction in \cite[Construction 3.4, Theorem 3.7]{weber2015operads} in the case of a unique color. Here the (previously singleton) set of operations outputting a vertex with arity $\ul N$ is $\O_N$ and the set of operations outputting an edge with arity $\vec 1 + \scdots{N} + \vec 1$, source operation $O \in \O_N$, and target operation $O' \in \O_N$ is the set of permutations $\sigma \in \Sigma_N$ such that $\O_\sigma(O) = O'$. Algebras for this algebraic familial monad on $\smcat$ are categories with functorial $N$-ary operations for each $O \in \O_N$ such that commutativity properties encoded in the symmetric operad $\O$ are weakened to symmetry isomorphisms in a category much like the relaxation from commutative monoids to symmetric monoidal categories in the previous example. Unlike \cref{operads}, we no longer need to assume $\O$ is $\Sigma$-free in order to model it as an algebraic familial monad on categories.
\end{example}


Lack and Street in \cite[Section 3]{lack2002formal} call monads in $\EM(\K)$ ``wreaths'' or ``extended distributive laws,'' noting that they generalize distributive laws in the sense that a monad $(c,m)$ in a bicategory $\K$ along with a monad morphism $n \colon (c,m) \to (c,m)$ and 2-cells in $\EM(\K)$ endowing this endomorphism with the structure of a monad amount to 2-cells in $\K$ of the form $mn \to nm$, $nn \to nm$, and $1 \to nm$ satisfying several equations. This is very nearly the data of a distributive law between monads $m,n$ in $\K$, but $n$ is not quite a monad in $\K$, but rather a ``monad up to $m$'' in a suitable sense (which in $\smcat$ amounts to being an endofunctor which lifts to a monad on the category of $m$-algebras). Nevertheless, much like a distributive law this data determines the structure of a monad on the composite $nm$ whose algebras agree with algebras for the monad derived from $n$ on the category of $m$-algebras.

In $\ccatsharp$, this result demonstrates that algebras for any algebraic familial monad $n$ on a category of $m$-algebras are just as well algebras for an ordinary familial monad $n \tri_c m$ on $c\set$.

\begin{example}[$\smc$ as the composite of $\sm$ and $\path$]\label{smwreath}
Applying this to the algebraic familial monad $\sm$, we recover the familial monad $\smc$ on graphs (as in, the monad on $g$ in $\ccatsharp$) from \cref{smc}. In particular it is straightforward to check from the definitions that $\smc = \sm \tri_g \path$, and as expected both $\sm$ as an algebraic familial monad on $\path$-algebras and $\smc$ as an ordinary familial monad on $g\set$ have symmetric monoidal categories as algebras. This example turns out to be an ordinary distributive law in $\ccatsharp$, which will generally be the case for cartesian monads such as $\sm$.
\end{example}


\printbibliography

\end{document}